\documentclass[10pt, oneside, reqno]{amsart}   		
\usepackage{geometry}                		
\usepackage{graphicx}				
\usepackage{amssymb, amsmath, amsthm, latexsym, fullpage}
\RequirePackage[dvipsnames,usenames]{xcolor}
\usepackage{enumitem}
\usepackage{hyperref}
\usepackage{diagbox}
\hypersetup{
bookmarks,
bookmarksdepth=3,
bookmarksopen,
bookmarksnumbered,
pdfstartview=FitH,
colorlinks,backref,hyperindex,
linkcolor=Sepia,
anchorcolor=BurntOrange,
citecolor=MidnightBlue,
citecolor=OliveGreen,
filecolor=BlueViolet,
menucolor=Yellow,
urlcolor=OliveGreen
}
\usepackage{xspace}
\usepackage{rotating}
\usepackage{verbatim}
\usepackage{tikz-cd}
\theoremstyle{plain}

\makeatletter
\newcommand{\xRightarrow}[2][]{\ext@arrow 0359\Rightarrowfill@{#1}{#2}}
\makeatother

\newtheorem{proposition}{Proposition}[section]
\newtheorem{theorem}[proposition]{Theorem}

\newtheorem{lemma}[proposition]{Lemma}
\newtheorem{corollary}[proposition]{Corollary}

\newtheorem{rem}[proposition]{Remark}
\newtheorem{example}[proposition]{Example} 
\newtheorem*{eg*}{Example}
\newtheorem*{cor*}{Corollary}
\newtheorem*{claim*}{Claim}

\DeclareMathOperator{\Ima}{Im}
\DeclareMathOperator{\rank}{rank}

\DeclareMathOperator{\ev}{ev}

\DeclareMathOperator{\disc}{disc}

\newcommand{\isom}{\cong} 				
\newcommand{\rarr}{\rightarrow}			
\newcommand{\Rarr}{\Rightarrow}			

\newcommand{\Lrarr}{\Longrightarrow}

\newcommand{\wt}{\widetilde}
\newcommand{\wh}{\widehat}
\newcommand{\hrarr}{\hookrightarrow}

\newcommand{\Grr}{\varrho}

\renewcommand{\AA}{\mathbb{A}}						
						
\newcommand{\CC}{\mathbb{C}}

\newcommand{\QQ}{\mathbb{Q}}

\newcommand{\ZZ}{\mathbb{Z}}

\title{Spaces of polynomials related to multiplier maps}
\author{Zhaoning Yang}
\date{\today}							
\address{Department of Mathematics, Johns Hopkins University, Baltimore, MD, 21218}
\email{zyang32@math.jhu.edu}

\begin{document}

\begin{abstract}
Let $f(x)$ be a complex polynomial of degree $n$. We attach to $f$ a $\CC$-vector space $W(f)$ that consists of complex polynomials $p(x)$ of degree at most $n - 2$ such that $f(x)$ divides $f''(x)p(x) - f'(x) p'(x)$. $W(f)$ originally appears in Yuri Zarhin's solution towards a problem of dynamics in one complex variable posed by Yu.~S.~Ilyashenko. In this paper, we show that $W(f)$ is nonvanishing if and only if $q(x)^2$ divides $f(x)$ for some quadratic polynomial $q(x)$. Then we prove $W(f)$ has dimension $(n-1) - (n_1 + n_2 + 2N_3)$ under certain conditions, where $n_i$ is the number of distinct roots of $f$ with multiplicity $i$ and $N_3$ is the number of distinct roots of $f$ with multiplicity at least three. 
\end{abstract}

\maketitle




\section{Definitions, notation, and statements} \label{sec: intro}

We write $\CC$ for the field of complex numbers, $\CC[x]$ for the ring of one variable polynomials with complex coefficients, and $P_s \subseteq \CC[x]$ for the subspace of polynomials of degree at most $s$. All vector spaces we consider are over $\CC$ unless otherwise stated. Throughout the paper $f(x) \in \CC[x]$ is a polynomial of degree $n \geq 4$. Let $R_i(f)$ be the set of distinct roots of $f(x)$ with multiplicity $i \geq 1$. Suppose that $n_i$ is the cardinality of $R_i(f)$ and $N_3$  is the number of distinct roots of $f$ with multiplicity at least three. 

We are interested in the space $W(f) \subseteq P_{n-2}$,  for which every $p(x) \in W(f)$ satisfies the condition that $p(x)$ divides $f''(x) p(x) - f'(x) p'(x)$. As a subspace of $P_{n-2}$, we wish to compute the dimension of $W(f)$ for various $f(x)$. The following assertions are main results of this paper.  

\begin{theorem} \label{T: mainthm}  \upshape Let $r  = r_f := n -2- (n_2 + 2N_3)$ and $\mu = \mu_f := r + 1 - n_1$. 
\begin{enumerate}[label=\textnormal{(\roman*)}]
	\item $\mu \leq \dim [ W(f) ] \leq r$ \label{item: bdsdim}
	\item If $0 \leq n_1 \leq 3$ or $r \geq 2n_1 -2$, then $\dim[ W(f) ]  = \mu$. \label{item: equdim}
 	\item If $n_1 = r = 4$, and $f(x)$ has \underline{at least two distinct multiple roots}, then $\dim[ W(f) ] = 2 > \mu =1$. \label{item: noneg}
\end{enumerate}
\end{theorem} 

Note that by rewriting $n$ as $n_1 + 2n_2 + \dots + m \cdot n_m + \dots$, we obtain 
\[
	\textstyle	r = -2 + \sum_{i \geq 1} i n_i - (n_2 + 2N_3) = -2 + n_1 + n_2 + n_3 + \sum_{i \geq 4} (i-2) n_i.
\]
Using the lower bound of $\dim [ W(f) ]$ in Theorem~\ref{T: mainthm}.\ref{item: bdsdim}, we deduce if $W(f)$ vanishes then $r \leq n_1 -1$. Combining with the above expression of $r$, we obtain
\[
 	\textstyle	 n_2 + n_3 + \sum_{i  \geq4 } (i-2) n_i \leq 1.
\]
Since $n_i$ are nonnegative integers, we deduce from the above inequality that $(n_2, n_3) \in \{ (0, 0), (0, 1) , (1, 0) \}$ and $n_i  = 0$ for all $i \geq 4$. This condition is equivalent to saying that for all quadratic polynomials $q(x) \in \CC[x]$, $q(x)^2$ does not divide $f(x)$. So by taking the contrapositive, we prove the following corollary.

\begin{corollary} \label{C: maincor} If there exists a quadratic $q(x) \in \CC[x]$ such that $q^2(x)$ divides $f(x)$, then $W(f)$ is nonzero.
\end{corollary}

In fact, Zarhin showed that \cite[Theorem~1.5.(ii)]{Zar12} the nonvanishing condition of $W(f)$ in Corollary~\ref{C: maincor} is sufficient. Suppose $f$ is monic with distinct roots $\alpha_1, \dots, \alpha_n$. We consider a holomorphic map $M: P_n \rarr \CC^n$ defined via $f(x) \longmapsto ( f'(\alpha_1), \dots, f'( \alpha_n) )$. At first, Zarhin used $W(f)$ to show \cite[Theorem 1.1]{Zar12} that the rank of the tangent map $d M$ is $n-1$ at all points of $P_n$. Based on these results, Zarhin proved \cite[Theorem 1.6]{Zar13} that the multipliers of any $n-1$ distinct periodic orbits, considered as multi-valued algebraic maps on $P_n$, are algebraically independent over $\CC$ under certain conditions. This question about algebraic independence was asked by Yu.~S.~Ilyashenko in the context of studying the Kupka-Smale property for volume-preserving polynomial automorphisms of $\CC^2$. (see \cite{BHI05, Gor16} for a detailed discussion).

Besides the role that $W(f)$ contributes to the proof of algebraic independence of multipliers, both \cite[Theorem~1.5]{Zar12} and alternative solutions of \cite[Theorem 1.1]{Zar12} provided by Victor S.~Kulikov \cite[\S4]{Zar12} and Elmer Rees \cite[\S1]{Rees15} show that $W(f)$ has some remarkable properties and so it might be of independent interest. As a byproduct of Theorem~\ref{T: mainthm}, we answer Rees's conjecture \cite[\S2]{Rees15} that the rank of $dM$ at $f$ is $n_1$ for arbitrary $f(x)$ allowing multiple roots. 

Let us compute $\dim [ W(f) ]$ for small $n$. In the following examples, we only consider nonzero $W(f)$s. By Corollary~\ref{C: maincor} and \cite[Theorem~1.5.(ii)]{Zar12}, $f(x) = q(x)^2 h(x)$ for a polynomial $h(x)$ and a quadratic $q(x) \in \CC[x]$. If $g(x) \in \CC[x]$ then  we write $R(x) = R(f, g)(x) := f''(x) g(x) - f'(x) g'(x)$. Notice that the defining condition $f(x) \mid R(f, p)(x)$ of $W(f)$ is preserved under an affine change of coordinates defined as $x \mapsto ax + b$ for any $a, b \in \CC, a \ne 0$. So by varying two parameters, we can assume $q(x) = x^2$ or $x^2 -1$, which depends on whether $q(x)$ has multiple roots.

\begin{example}[quartic polynomial] \label{eg: degf=4} \upshape If $n = 4$, then $W(f) = \CC \cdot q(x)^2 \Rarr \dim [ W(f) ] = 1$ from \cite[Theorem~1.5.(iv)]{Zar12}. Notice that $f(x) = c q(x)^2$ for some nonzero $c \in \CC$ by Corollary~\ref{C: maincor}. It follows that $(n_2, n_4) = (2, 0)$ or $(0, 1)$ and $n_i = 0$ elsewhere. In both cases $r = 0$, which implies that $\mu =1 = \dim [ W(f) ]$. 
\end{example}

\begin{example}[quintic polynomial] \label{eg: degf=5} \upshape Let $n = 5$, we have $\deg h = n - 2 \deg q = 1$. Assume $h$ is monic, then $h(x) = x - c$ for some $c \in \CC$.

Suppose $q(x) = x^2$. If $c \ne 0$, then $f(x) = x^4(x-c)$ contains exactly one simple root and one root of multiplicity four. From the condition $f \mid R$, one checks that elements in $W(f)$ are scalar multiples of $ x( 6x -  5c)$. So $\dim [ W(f) ] = 1$. If $c = 0$, then $f(x) = x^5$. Again, calculating from $f \mid R$ shows that $W(f) = \{ x^2 g(x) :  \deg g = 1 \}$. Therefore $\dim [ W(f) ] = 2$. 

Now suppose $q(x ) = x^2 -1$, we check that $\dim [ W(f) ] = 1$. If $q(c) \ne 0$, $f(x)$ contains exactly one simple and two double roots. The condition $f \mid R$ is equivalent to saying that $p(x) \in W(f)$ is a scalar multiple of $(x^2 - 1)( 6c x - 5c^2 - 1)$. On the other hand, if $q(c) = 0$, $f(x)$ has exactly one double root and one root of multiplicity three. From $f \mid R$, $p(x) \in W(f)$ if and only if it is a scalar multiple of $(x^2 - 1) (x - c)$. 
\end{example}

Complete results on related information of $\dim [W(f) ]$ for small $n$ values are summarized in Table~\ref{table: degf=4/5/6}. 

\begin{table}[ht]						
\caption{quartic, quintic and sextic polynomial $f(x)$} 	
\centering 										
	\begin{tabular}{|c | c c | c c c c| c c c c c c c|} 
		\hline
 		\diagbox[width=6em]{$$}{$\deg f$}	& 4& & & 5& & & & & &6 & & & \\
		\hline
		$n_2$ 			&0 &2& 0&0 &2 &1 & 0& 1& 0& 0& 3& 1& 2\\
		$N_3$ 			&1 &0 &1 &1 &0 &1 & 1& 1& 1& 1& 0& 1& 0\\
		$r$ 				&0 &0 &1 &1 &1 &0 & 2& 1& 2& 2& 1& 1& 2\\
		$n_1$ 			&0 &0 &1 &0 &1 &0 & 0 &0 &1 &2 &0 &1 &2 \\
		$\mu$			&1 &1 &1 &2 &1 &1 & 3& 2& 2& 1& 2& 1& 1 \\
		\hline
		$\dim[ W(f) ]$ 		&1 &1 & 1&2 &1 &1 & 3& 2& 2& 1& 2& 1& 1\\
		\hline
	\end{tabular}
\label{table: degf=4/5/6}
\end{table}

\subsection{Notation}  \label{subsec: nota} Let $R(f)$ be the set of distinct roots of $f(x)$.  Let $R_i(f) \subseteq R(f)$ be the set of roots of $f(x)$ with multiplicity exactly $i$, as before $n_i := | R_i( f ) |$. In the course of the proof of $\dim [ W(f) ] = \mu$ (Theorem~\ref{T: mainthm}.\ref{item: equdim}), a key step is to group elements in $R(f)$ according to different multiplicities. So we divide $R(f)$ into the following three sets
\[
	\textstyle \alpha := R_1(f),\;  \beta := R_2(f), \text{ and } \gamma := \bigcup_{m \geq 3} R_m(f).
\]
Notice that $N_3 = | \gamma |$ and elements $\alpha_i \in \alpha, \beta_j \in \beta$ are simple and double roots of $f(x)$ respectively. Since multiplicities of elements in $\gamma$ may change, we set $k_s$ to be the multiplicity of each elementa $\gamma_s \in \gamma$. 

Let $f_i(x) := \prod_{ \Grr \in R_i(f) } (x - \Grr)$,  and $f_{\alpha} (x) := f_1(x), f_{\beta} ( x)  := f_2(x), f_{ \gamma } (x) := \prod_{m \geq 3} f_m(x)$. For any $p(x) \in W(f)$, the advantage of dividing $R(f)$ into $\alpha, \beta$, and $\gamma$ is that we can view the condition $f(x) \mid R(f, p)(x)$ as three separate relations:
\[
	\textstyle  f_{\beta} (x)^2 \mid R(f, p) (x) ,  \;   \prod_{i \geq 3} [ f_i(x) ]^i  \mid R(f, p)(x), \text{ and }  f_{\alpha} (x) \mid R(f, p)(x). \tag{\ref{subsec: nota}.1}  \label{equ: divcond}
\]
In \S\ref{sec: subW23}, we show the first two conditions of (\ref{equ: divcond}) are equivalent to $f_{\beta} (x)[  f_{\gamma}(x) ]^2 \mid p(x)$. Let $p_{\alpha} ( x ) := p(x) / [ f_{\beta} (x) f_{\gamma} (x)^2 ]$, the last condition of (\ref{equ: divcond}) can be reduced to an interpolation of the rational function $p_{\alpha} '(x) / p_{\alpha} ( x ) $ at simple roots of $f$. More precisely, let 
\[
	d(x) = d_f(x) := \frac{ f_{\alpha}''(x) }{ f_{\alpha}'(x) } + \sum_{i = 1}^{n_2} \frac{ 3 }{x - \beta_i} + \sum_{s = 1}^{N_3} \frac{ 2(k_s - 1) }{x - \gamma_s}, \tag{\ref{subsec: nota}.2} \label{equ: d_f}
\]
the following lemma is an equivalent version of (\ref{equ: divcond}).

\begin{lemma} \label{L: importL} Let $p(x)$ be a polynomial of degree at most $n-2$. Then $p(x) \in W(f)$ if and only if 
\begin{enumerate}
	\item $f_{\beta} ( x ) f_{\gamma}^2( x )$ divides $p (x  )$; \label{item: multRootdiv}
	\item The function $p_{\alpha} ' ( x) - d( x) p_{\alpha}( x)$ at all simple roots $\alpha_i$ of $f(x)$. \label{item: dinterp}
\end{enumerate} 
\end{lemma}

Let $k, s$ be positive integers, given two $s$-tuples $\eta = (\eta_1, \dots, \eta_s), \omega = ( \omega_1, \dots, \omega_s )$ in $\CC^s$ with $\omega_i \ne \omega_j$. Mimicking (\ref{L: importL}.\ref{item: dinterp}) in the previous lemma, we define a polynomial space
\[
		Z( \eta, \omega , s, k) := \{ p \in \CC[x] : \deg p \leq k, p' ( \omega_i ) = \eta_i p( \omega_i ) \; \forall \; 1 \leq i \leq s \}.  \tag{\ref{subsec: nota}.3} \label{stat: defnZ}
\] 
We mostly consider the case $Z(n_1, r) = Z(\delta, \alpha, n_1, r)$ where $\delta = ( d(\alpha_1), \dots, d( \alpha_{n_1} ) ), \alpha = (\alpha_1, \dots, \alpha_{n_1} )$ are two $n_1$-tuples. Viewing $W(f)$ as a special type of $Z(\eta, \omega, k ,s )$ is the motivation of (\ref{stat: defnZ}) and another important step to prove Theorem~\ref{T: mainthm}. 

Note $\deg p \leq \deg f - 2$ if and only if $\deg p_{\alpha} \leq r$. Furthermore if $p, q \in W(f)$ and $\lambda \in \CC$
\[
	(p+ \lambda q)_{\alpha} ( x ) = \frac{ p ( x ) + \lambda q(x) }{ f_{\beta} ( x ) f_{\gamma} ( x )^2 } =  \frac{ p(x) }{   f_{\beta} ( x ) f_{\gamma} ( x )^2  } + \lambda \frac{ q( x ) }{  f_{\beta} ( x ) f_{\gamma} ( x )^2  } =p_{\alpha} (x) + \lambda q_{\alpha} ( x).
\]
By Lemma~\ref{L: importL}, the map $\phi: W(f) \rarr Z( \delta, \alpha , n_1, r )$ defined via $p(x) \mapsto p_{\alpha}(x)$ is a linear isomorphism. The map $\psi: p(x) \longmapsto f_{\beta} ( x ) f_{\gamma}(x)^2 p(x)$ from $Z( n_1, r)$ to $W(f)$ is the inverse of $\phi$. We formally state this result as the following corollary.

\begin{corollary}  \label{C: CimportL} $W(f)$ and $Z( n_1, r )$ are isomorphic (via $\phi$). In particular, $\dim [ W(f) ] = \dim [ Z(n_1, r) ]$.
\end{corollary}

Note that Lemma~\ref{L: importL} also partially checks Corollary~\ref{C: maincor}. Let $r \geq 2n_1$ (i.e. $n -2 \geq 2n_1 + n_2 + 2N_3$.), then $p_0(x) := f_{\alpha}^2(x) f_{\beta} (x) f_{\gamma}^2(x)$ is of degree at most $n-2$. Since $p_0(x)$ has double roots at all $\alpha_i$s, by Lemma~\ref{L: importL} $p_0(x)$ is an nonzero element in $W(f)$.


\section{Study of $W(f)$ for $f$ modulo simple roots} \label{sec: subW23} 

\subsection{Proof of Theorem~\ref{T: mainthm}.\ref{item: equdim} (when $n_1 =0$)}  \label{subsec: n_1=0=>dim=mu} 

We claim Theorem~\ref{T: mainthm}.\ref{item: equdim} in the case $n_1 =0$. We set $\wt{ f }_{\gamma} (x ) :=  \prod_{m \geq 3} [ f_m(x) ]^m$. Recall $f_{\alpha}(x) = f_1(x), f_{\beta}(x) = f_2(x)$ in \S\ref{subsec: nota}. Since $f(x) = \prod_{m \geq 1} [ f_m(x) ]^m$, we can say
\[
	\textstyle    \wt{ f }_{\gamma} (x ) =  \prod \nolimits_{m \geq 3} [ f_m(x) ]^m = f(x) / [  f_{\alpha} ( x )  f_{\beta} (x)^2 ]. \tag{\ref{subsec: n_1=0=>dim=mu}.2}
\]
We can view $W(f)$ as intersections of the following vector spaces.
\[
	\begin{cases}
		W(f, \alpha) &:=  \{ p(x) \in \CC[x] : \deg p \leq (n - 2), f_{\alpha}(x) \text{ divides } R(f, p)(x) \};	\\
		W(f, \beta) &:=  \{p(x) \in \CC[x] : \deg p \leq (n - 2), f_{\beta}^2(x) \text{ divides } R(f, p)(x) \};	\\ 
		W(f, \gamma) &:=  \{ p(x) \in \CC[x] :  \deg p \leq (n - 2), \wt{f}_{\gamma}(x) \text{ divides }R(f, p)(x) \}.
	\end{cases}
	\tag{\ref{subsec: n_1=0=>dim=mu}.3}
	\label{equ: W(-,f)}
\]

\begin{rem} \label{R: R_1=0} \upshape In particular, if $R_1(f) = \varnothing$ (i.e. $f_{\alpha}(x) \equiv 1$) then $W(f, \alpha)$ is the space of all polynomials with degree at most $n - 2$, which means $W(f) = W(f, \beta) \cap W(f, \gamma)$.
\end{rem}

The advantage of doing this is that spaces introduced in (\ref{equ: W(-,f)}) are much easier to characterize as the following theorem shows. 

\begin{theorem} \label{T: thmMRoot} Let $p(x)$ be a polynomial of degree at most $n-2$. 
	\begin{enumerate}	
		\item $p(x) \in W(f, \beta)$ if and only if $f_{\beta}(x)$ divides $p(x)$;
		\item $p(x) \in W(f, \gamma)$ if and only if $f_{\gamma}(x)^2$ divides $p(x)$.
	\end{enumerate}
\end{theorem}

Since elements in the set $\Gamma := \{ (x- \beta_j)^2, ( x- \gamma_s)^{k_s} : 1 \leq j \leq n_2, 1 \leq s \leq N_3 \}$ are pairwise coprime, the condition $f_{\beta}(x)^2 \wt{ f }_{\gamma} ( x ) \mid R(f, p) (x )$ can be considered separately and equally as $g(x) \mid R(f, p)( x )$ for all $g \in \Gamma$. So by induction, Theorem~\ref{T: thmMRoot} reduces to the following lemma. 

\begin{lemma} \label{L: lemMRoot} Let $p(x) \in W(f)$. 
	\begin{enumerate}
		\item If $\beta \in R_2(f)$, then $(x - \beta)^2$ divides $R(f, p)(x)$ if and only if $(x - \beta)$ divides $p(x)$;
		\item If $\gamma \in R_m(f) \; (m \geq 3)$, then $(x - \gamma)^m$ divides $R(f, p)(x)$ if and only if $(x - \gamma)^2$ divides $p(x)$.
	\end{enumerate}
\end{lemma}

\begin{proof}
Let $x = \beta$ be a double root of $f(x)$. Then $f(\beta) = f'(\beta) = 0$ and $f''(\beta) \ne 0$. First we compute the derivative of $R(f, p)(x)$,
\[
		R(f,p)'(x) = [ f'''(x)p(x) + f''(x)p'(x) ] -[ f''(x)p'(x) + f'(x)p''(x) ]  = f'''(x)p(x) - f'(x)p''(x).
\]
So it follows from the above formula of $R(f, p)(x)$ and $R'(f, p)(x)$ that $R(f, p)(\beta) = f''(\beta)p(\beta)$, $R'(f, p)(\beta) = f'''(\beta)p(\beta)$, and
	\[
		(x - \beta)^2 | R(f, p)(x) \iff R(f, p)(\beta) = R'(f, p)(\beta) = 0.
	\]
Since $f''(\beta) \ne 0$, $R(f, p)(\beta) = 0 \iff p(\beta) = 0$. Combining it with $R(f, p)'(\beta) = f'''(\beta) p(\beta)$ we have 
	\[
		R(f, p)(\beta) = R(f, p)'(\beta) = 0 \iff p(\beta) = 0.
	\]
Similarly, assume $\gamma \in R_m(f)$ with $m \geq 3$. This means that $f(x) = (x - \gamma)^m \wt{f}(x)$ where $\wt{f}(\gamma) \neq 0$. So, we can express $f'(x)$ and $f''(x)$ using $\wt{f}(x), \wt{f}'(x)$, and $\wt{f}''(x)$.
	\begin{align*}
		f'(x) &= m (x - \gamma)^{m  -1}\wt{f}(x) + (x - \gamma)^m \wt{f}'(x), \\
		f''(x) &= m(m  - 1)(x - \gamma)^{m - 2}\wt{f}(x) + 2m(x - \gamma)^{m - 1}\wt{f}'(x) + (x - \gamma)^m\wt{f}''(x).
	\end{align*}
We denote $Q(x) := R(f, p)(x)/(x - \gamma)^{m - 2}$. Plugging $f'(x)$ and $f''(x)$ into $R(f, p)(x)$, we can express $Q(x)$ in terms of $\wt{f}(x)$
\begin{align*}
	Q(x) = \Big[ m(m-1)\wt{f}(x) + 2m(x - \gamma)\wt{f}'(x) + (x - \gamma)^2\wt{f}''(x) \Big] p(x) - (x - \gamma)p'(x) \Big[ m\wt{f}(x) + (x - \gamma) \wt{f}'(x) \Big].
\end{align*}
Next, we express $Q(x)$ in the monomial basis $\{ 1, (x - \gamma), (x - \gamma)^2 \}$,
\begin{align*}
	Q(x) = m (m - 1)\wt{f}(x)p(x) + m(x - \gamma)\Big[ 2\wt{f}'(x)p(x) - \wt{f}(x)p'(x) \Big] +(x - \gamma)^2R(\wt{f}, \; p)(x).
\end{align*}
Explicit substitution shows that $Q(\gamma) = m(m - 1)\wt{f}(\gamma)p(\gamma)$. Since $m \geq 3$, $m(m-1) \ne 0$ and $m(m-2) \ne 0$. Since $\wt{ f } ( \gamma ) \ne 0$, $Q(\gamma) = 0 \iff p(\gamma) = 0$. In addition
\begin{align*}
	Q'(x) = &m(m - 1)\Big[ \wt{f}'(x)p(x) + \wt{f}(x)p'(x) \Big] + m \Big[ 2\wt{f}'(x)p(x) - \wt{f}(x)p'(x) \Big] \\
			&+ m(x - \gamma)\Big[ 2\wt{f}''(x)p(x) + \wt{f}'(x)p'(x) - \wt{f}(x)p''(x) \Big] + 2(x - \gamma) R(\wt{f}, \; p)(x) + (x - \gamma)^2R'(\wt{f}, \; p)(x).
\end{align*}
Evaluating $Q'(x)$ at $x = \gamma$, we obtain that $Q'(\gamma) = m(m + 1)\wt{f}'(\gamma)p(\gamma) + m(m - 2)\wt{f}(\gamma)p'(\gamma)$. So if $Q(\gamma) = Q'(\gamma) = 0$, we have $p(\gamma) = 0$ and $Q'(\gamma) = m(m - 2)\wt{f}(\gamma)p'(\gamma) = 0$. Since $m(m-2) \ne 0$, $\wt{f}(\gamma) \ne 0$ implies $p'(\gamma) = 0$. Conversely, $p(\gamma) = p'(\gamma) = 0$ also implies $Q(\gamma) = Q'(\gamma) = 0$. So we prove $(x - \gamma)^2 \bigr| Q(x) \iff (x - \gamma)^2 \bigr| p(x)$. By construction of $Q(x)$, $(x - \gamma)^m \mid R(f, p)(x) \iff (x - \gamma)^2 \mid Q(x)$. Hence $(x - \gamma)^m$ divides $R(f, p)(x)$ if and only if $(x - \gamma)^2$ divides $p(x)$.
\end{proof}

\begin{corollary} \label{C: n_1=0} If $R_1(f) = \varnothing$ (i.e., $n_1 =0$) then $W(f)  = \{ p(x) \in \CC[x] \mid \deg p \leq n - 2, \text{ and } f_{\beta} f_{\gamma}^2 \text{ divides } p \}$.
In particular $\dim[ W(f) ] = r + 1$. 
\end{corollary}

\begin{proof} By Theorem~\ref{T: thmMRoot} since $f_{\beta}$ and $f_{\gamma}^2$ are relatively prime,
\[
	p(x) \in W(f, \beta) \cap W(f, \gamma) \iff f_{\beta}(x) \mid p(x) \text{ and } f_{\gamma} ( x )^2 \mid p(x) \iff  f_{\beta}(x)f_{\gamma}^2(x) \text{ divides } p(x).
\]
Let us prove Theorem~\ref{T: mainthm}.\ref{item: equdim} under the case $n_1=0$. By Remark~\ref{R: R_1=0}, if $n_1 = 0$ then
\[
	W(f) = W(f, \beta) \cap W(f, \gamma) = \{ p(x) \in \CC[x] \mid \deg p \leq n - 2, \text{ and } f_{\beta} f_{\gamma}^2 \text{ divides } p \}.
\]
In other words, elements in $W(f)$ are of the form $h(x) f_{\beta} ( x ) f_{\gamma} ( x)^2$ where $h(x)$ is a polynomial of degree at most $n-2 - (n_2 + 2N_3)  =r$. So $\dim [ W(f) ] = r +1$. 
\end{proof}

\begin{rem} \label{rem: maincorn_1=0} \upshape Recall in Theorem~\ref{T: mainthm}, $r = n- 2 - (n_2 + 2 N_3)$. If we write $n$ as $n_1 + 2n_2 + k_1 + \dots + k_{N_3}$, then we can rewrite $r$ as
\[
	\textstyle r =  n_1 + (n_2 - 2) + \sum_{s = 1}^{N_3} (k_s - 2). \tag{\ref{subsec: n_1=0=>dim=mu}.4} \label{equ: r}
\]
If $f(x)$ is divisible by the square of a quadratic polynomial, we have either $n_2 \geq 2$ or $N_3 \geq 1$ together with some $k_s \geq 4$. By (\ref{equ: r}), $r \geq n_1$. So $\dim [ W(f) ] = r + 1 \geq 1 \Rarr W(f)$ is nonvanishing. So we prove Corollary~\ref{C: maincor} when $n_1 =0$. 
\end{rem}

\subsection{Proof of Lemma~\ref{L: importL}} \label{subsec: reform} 

Recall by (\ref{equ: W(-,f)}), $W(f) = W(f, \alpha) \cap W(f, \beta) \cap W(f, \gamma)$. Also, Theorem~\ref{T: thmMRoot} says $p(x) \in W(f, \beta) \cap W(f, \gamma)$ if and only if $p_{\alpha}(x) = p(x) / [f_{\beta}(x) f_{\gamma}^2(x)]$ is a polynomial. From these remarks, Lemma~\ref{L: importL} reduces to the following claim.

\begin{claim*}  \label{L: equiv} Let $p(x)$ be a polynomial of degree at most $n-2$. Then $p(x)  \in W(f, \alpha)$ if and only if $d(x)p_{\alpha}(x) - p_{\alpha}'(x)$ vanishes at all simple roots of $f$. 
\end{claim*}

\subsection*{Proof of Claim} Recall from (\ref{subsec: n_1=0=>dim=mu}.2) that $\wt{ f }_{\gamma}(x) = f(x) / [ f_{\alpha} ( x ) f_{\beta} (x)^2 ] =  \prod_{i = 1}^{N_3} (x - \gamma_i)^{k_i}.$ If $g(x) = \prod_{i=1}^n (x - \omega_i)$ then
\[
	 \frac{ g'(x) }{ g(x) } = \sum_{i = 1}^n \frac{ 1 }{ x - \omega_i } \text{ and } \frac{ g''( \omega_i ) }{ g'( \omega_i ) } = \sum_{j \ne i} \frac{ 2 }{ \omega_i - \omega_j }. \tag{\ref{subsec: reform}.1} \label{equ: polyalg}
\]
Using these facts, we can rewrite $d(x)$ from (\ref{equ: d_f}) as the following:
\[
	d(x) =  \frac{f''_{\alpha}(x)}{f'_{\alpha}(x)} + 3 \frac{f'_{\beta}(x)}{f_{\beta}(x)} + 2\frac{ \wt{f}_{\gamma} '(x)}{\wt{f} _{\gamma}(x)} - 2 \frac{f'_{\gamma}(x)}{f_{\gamma}(x)}. \tag{\ref{subsec: reform}.2}  \label{equ: d_fReform}
\]
We set $\wt{f}_{\beta} = f_{\beta}^2$, $p_{\gamma} = f_{\gamma}^2$ and rewrite $f$, $p$ as $f = f_{\alpha} \cdot \wt{ f} _{\beta}    \cdot  \wt{ f} _{\gamma}$, $p = p_{\alpha} \cdot f_{\beta} \cdot p_{\gamma}$. It follows that
\begin{align*}
	p' &= p_{\alpha}' f_{\beta} p_{\gamma} + p_{\alpha} f_{\beta}' p_{\gamma} + p_{\alpha} f_{\beta} p_{\gamma}',	\; f' = f_{\alpha}' \wt{ f}_{\beta} \wt{ f}_{\gamma} + f_{\alpha} ( \wt{ f}_{\beta}' \wt{ f}_{\gamma} + \wt{ f}_{\beta} \wt{ f}_{\gamma}'), 	 \\
	f'' &= f_{\alpha}'' \wt{ f}_{\beta} \wt{ f}_{\gamma} + 2f_{\alpha}' ( \wt{ f}_{\beta}' \wt{ f}_{\gamma} + \wt{ f}_{\beta} \wt{ f}_{\gamma}' ) + f_{\alpha} (\wt{ f}_{\beta}'' \wt{ f}_{\gamma} + \wt{ f}_{\beta} \wt{ f}_{\gamma}'').
\end{align*}
Let $F$ be the polynomial $R(f, p)  - f_{\alpha} [ p (\wt{ f}_{\beta}'' \wt{ f}_{\gamma} + \wt{ f}_{\beta} \wt{ f}_{\gamma}'') - ( \wt{ f}_{\beta}' \wt{ f}_{\gamma} + \wt{ f}_{\beta} \wt{ f}_{\gamma}') p' ]$. While evaluating $R(f, p)$ at $x = \alpha_i$, we may disregard all terms that are divisible by $f_{\alpha}(x)$. So, $F$ and $R(f, p)$ vanish at all $x = \alpha_i$ simutaneously. Plugging in $\wt{ f}_{\beta} = f_{\beta}^2, \wt{ f}_{\beta}' = 2 f_{\beta} f_{\beta}'$ into $F(x)$, we obtain 
\begin{align*}
	F &= \Big[ f_{\alpha}'' f_{\beta}^2 \wt{ f}_{\gamma} + 2 f_{\alpha}' ( 2 f_{\beta} f_{\beta}' \wt{ f}_{\gamma} + f_{\beta}^2 \wt{ f}_{\gamma}' ) \Big] p_{\alpha} f_{\beta} p_{\gamma} 
					- f_{\alpha}' f_{\beta}^2 \wt{ f}_{\gamma} \Big[ p_{\alpha}' f_{\beta} p_{\gamma} + p_{\alpha} f_{\beta}' p_{\gamma} + p_{\alpha} f_{\beta} p_{\gamma}' \Big].   \tag{\ref{L: equiv}.3} \label{equ: Freform}
\end{align*}
Let $G(x)$ be the polynomial $f_{\alpha}'(x) f_{\beta}^3(x) p_{\gamma}(x) \wt{ f}_{\gamma}(x)$. From (\ref{equ: d_fReform}), $F/ G = d (x) p_{\alpha} (x) - p_{\alpha}' (x)$ when we divide $G$ on both sides of (\ref{equ: Freform}). Note $G( \alpha_i ) \ne 0$ for all $1 \leq i \leq n_1$. As rational functions, it is equivalent to say $F$ or $F / G$ vanish at all $\alpha_i$. In other words, $R(f, p)( \alpha_i ) = 0$ if and only if $d(\alpha_i) p_{\alpha} ( \alpha_i) - p'_{\alpha} ( \alpha_i ) = 0$ for all $1 \leq i \leq n_1$. 


\section{Basic properties of $Z( \eta, \omega ; s , k)$} \label{sec: HSpace}

\begin{proposition} \label{P: natEmbed} Let $\eta, \omega$ be points in $\CC^s$ with $\omega_i \ne \omega_j$ for all $i \ne j$ and assume $s' \leq s, k' \leq k$. If $\eta' = (\eta_1, \dots, \eta_{s'}), \omega' = (\omega_1, \dots, \omega_{s'} )$ are points in $\CC^{s'}$ then
\begin{enumerate}
\item If $k'' \geq k$, then $\dim [ Z(\eta, \omega; s, k'') ] \leq \dim [Z( \eta, \omega ; s , k) ] + (k'' - k )$;
\item There is a chain of vector space embeddings $Z(\eta, \omega; s, k') \hrarr  Z( \eta, \omega ; s , k)  \hrarr  Z(\eta', \omega'; s', k)$. Let $i_{k'k}, i_{ss'}$ be the first and second map respectively. They are both natural inclusions.
\end{enumerate}
\end{proposition}

\begin{proof} For part (1), in $Z = Z(\eta, \omega, s, k)$ if we replace $k$ by $k+1$, then the dimension of the resulting space $\wt{Z} = Z( \eta, \omega, s, k+1)$  would jump up at most one. Namely, if $Z$ and $\wt{ Z }$ do not coincide then $\wt{Z}$ is generated by $Z$ and a certain polynomial of degree $k+1$. 

For part (2), one may easily check that $Z(\eta, \omega; s, k')$ is a subspace of $Z( \eta, \omega ; s , k)$ and $Z( \eta, \omega ; s , k)$ is a subspace of $ Z(\eta', \omega'; s', k')$. We denote the corresponding inclusion maps by $i_{k'k}$ and $i_{ss'}$ respectively. 
\end{proof}

We will use the following obvious assertion. 

\begin{proposition} \label{P: ax+b} For $a, b \in \CC$ constants with $a \ne 0$, the map $\phi_{a, b}:  Z( \eta, \omega ; s , k) \rarr Z(\eta', \omega'; s, k)$ defined by $\phi_{a, b} (p(x) ) = p( a^{-1} (x - b) )$ is an invertible linear map where $\eta' =  a^{-1} \eta, \omega' = a \omega + b$.
\end{proposition} 

\begin{theorem}[Bounds] \label{T: bds4Z} If $k \geq s - 1$ then $k + 1 -s \leq \dim [ Z( \eta, \omega ; s , k) ] \leq k$.
\end{theorem} 

\begin{proof} Let $p(x) \in Z( \eta, \omega ; s , k)$, we can write $p(x) = a_k x^k + \dots + a_1x + a_0$. From (\ref{stat: defnZ}), the defining condition $p ' ( \omega_i ) = \eta_i p( \omega_i )$ is equivalent to $k a_k { \omega_i }^{k-1} + \dots + a_1 = a_k \omega_i^k + \dots + a_0$ for all $1 \leq i \leq s$. By treating $(a_0, \dots, a_k) \in \CC^{k+1}$ as variables, we can regard this as a homogeneous $s \times (k+1)$ linear system. So, we can write down the corresponding matrix $A$ with respect to the monomial basis $\{ 1, x, \dots, x^k \}$.
\[
	A = 
	\begin{pmatrix}
	\eta_1 &  \omega_1 \eta_1 - 1 & \dots & \omega_1^k \eta_1 - k \omega_1^{k -  1} \\
	\eta_2 & \omega_2 \eta_2  - 1 & \dots & \omega_2^k \eta_2 - k \omega_1^{k - 1} \\
	\vdots & \vdots  & \ddots &\vdots\\
	\eta_{s} & \omega_{s} \eta_{s} - 1 & \dots & \omega_{s}^k \eta_{s} - k \omega_{s}^{k - 1} \\
	\end{pmatrix}
	\tag{\ref{T: bds4Z}.1} \label{equ: assomatrix}
\]
Since $s \leq k + 1$, the number of columns in $A$ is always greater than or equal to the number of rows in $A$. Since $\dim [ Z( \eta, \omega ; s , k)] $ is the same as the corank of $A$, $k +1 =  \# \text{ columns of } A  = \rank A + \dim [  Z( \eta, \omega ; s , k) ]$. So it suffices to check $1 \leq \rank A \leq s$. First note that $\rank A \leq \min \{ \# \text{ columns of } A, \# \text{ rows of } A \} = \min \{ k + 1, s \} =  s$. On the other hand, by multiplying the first column of $A$ by $- \omega_1$ and adding it to the second column, the first entry of the resulting column becomes $-1 \ne 0$. Hence $\rank A \geq 1$. 
\end{proof}

By Remark~\ref{rem: maincorn_1=0}, Corollary~\ref{C: maincor}, and \cite[Theorem~1.5.(ii)]{Zar12},  if $W(f)$ is nonvanishing then
\[
	\textstyle r = n_1 + (n_2 -2 ) + \sum_{i=1}^{N_3} (k_i -2) \geq n_1.
\]
Applying Theorem~\ref{T: bds4Z} to $Z( \delta, \alpha, n_1, r ) \isom W(f)$, we get $\mu \leq \dim [ W(f) ] \leq r$. This verifies Theorem~\ref{T: mainthm}.\ref{item: bdsdim}. The matrix $A$ defined in (\ref{equ: assomatrix}) play an important role when we compute $\dim [ W(f) ]$. From now on, we call $A$, the \emph{\textbf{associated matrix}} attached to $Z( \eta, \omega, s, k)$. At first glance, $A$ shares similarity with the Vandermonde matrix, so we expect $A$ attains full rank under certain mild conditions. Here are some examples.

\begin{example} \label{E: eta=0} \upshape Let $\eta =(0, \dots, 0) \in \CC^s$, and $\widetilde{V}(\omega)$ be the matrix truncated from the second to the $(s + 1)$th columns in the associated matrix of $Z( 0, \omega, k, s )$. It's not hard to see that $\wt{V}(\omega)$ is obtained from the Vandermonde matrix $V( \omega)$ by multiplying $-j$ on the $j$th column for each $1 \leq j \leq s$. Therefore
\begin{align*}
	 \det \wt{V}(\omega) &=  \textstyle  s!(-1)^s\det V( \omega) = s!(-1)^s v_n(\omega) = s! (-1)^s \prod_{1 \leq i < j \leq s} ( \omega_j - \omega_i) \ne 0.
\end{align*}
where $v_n = \prod_{1 \leq i < j \leq n} (x_j - x_i)$ is the Vandermonde polynomial. Hence $\rank ( \widetilde{V}(\omega) )= s \Lrarr \rank A = s$. So $\dim Z( \eta, \omega ; s , k) = k + 1 - \rank A = k + 1 -s$. 
\end{example} 

\begin{example} \label{E: s=2} \upshape 
We use a brute-force calculation to check if $k \geq 3$, $\dim [ Z(\eta, \omega; 2, k) ]  =  k - 1 = \mu$. Since $k \geq 3$, the associated matrix $A$ has at least four columns. Our plan is to prove by contradiction. Suppose to the contrary, then Remark~\ref{rem: equiv} says $A$ does not have full rank. Let $A_1, A_2$ be the first and second row of $A$ respectively. Since $A$ is a $2 \times (k + 1)$ matrix, $A_1 = c A_2$ for some nonzero $c \in \CC$. By (\ref{equ: assomatrix}), 
\[
	 (\eta_1, \eta_1 \omega_1 - 1, \eta_1 \omega_1^2 - 2 \omega_1, \eta_1 \omega_1^3 - 3 \omega_1^2, \dots ) = c (\eta_2, \eta_2 \omega_2 - 1, \eta_2 \omega_2^2 - 2 \omega_2, \eta_2 \omega_2^3 - 3 \omega_2^2, \dots ).
\]
Equating the first entry from the above expression, we obtain $\eta_1 = c \eta_2$. Substituting $\eta_1 = c \eta_2$ in the second through fourth entries of the equation $A_1  = cA_2$, and we have
\begin{align}
	c\eta_2(\omega_1 - \omega_2) &= 1 - c, \tag{\ref{E: s=2}.1} \label{equ: 1-c} \\
	c\eta_2(\omega_1^2 - \omega_2^2) &= 2\omega_1 - 2c \omega_2, \tag{\ref{E: s=2}.2} \label{equ: s=2Eta} \\
	c\eta_2(\omega_1^3 - \omega_2^3) &= 3\omega_1^2 - 3c \omega_2^2.  \tag{\ref{E: s=2}.3}
\end{align}
We continue to show that (\ref{equ: 1-c}) and (\ref{equ: s=2Eta}) implies 
\[
	c = -1, \; \eta_1 + \eta_2 = 0, \text{ and } \eta_2(\omega_2 - \omega_1) = 2. \tag{\ref{E: s=2}.4} \label{stat: s=2}
\]
Let $R_2( \eta, \omega)$ be the right-hand side of $(\ref{equ: s=2Eta})$. Plugging $1-c$ from (\ref{equ: 1-c}) into $R_2(\eta, \omega)$, we obtain 
\begin{align*}
		R_2( \eta, \omega ) &= 2\omega_1 - 2c \omega_2 + (2 \omega_2 - 2 \omega_2) =  2(\omega_1 - \omega_2) + 2\omega_2(1 - c) = (\omega_1 - \omega_2)(2 + 2c \eta_2 \omega_2).
\end{align*}
Dividing $\omega_1 - \omega_2$ on both sides of (\ref{equ: s=2Eta}), we obtain $c\eta_2(\omega_1 + \omega_2) = 2 + 2c\eta_2\omega_2 \Lrarr c\eta_2(\omega_1 - \omega_2) = 2$. Combining this with (\ref{equ: 1-c}), we can say $1 - c = c \eta_2(\omega_1 - \omega_2) = 2 \Rarr c = -1$. In short, we obtain $\eta_2(\omega_2 - \omega_1) = 2$ by substituting $\eta_1 = c \eta_2 = - \eta_2$ in (\ref{equ: 1-c})

Let $L_2( \eta, \omega)$ be the left-hand side of (\ref{E: s=2}.3). Using (\ref{E: s=2}.4), we obtain 
\[
	L_2( \eta, \omega ) = -\eta_2(\omega_1 - \omega_2)(\omega_1^2 + \omega_1\omega_2 + \omega_2^2) = 2(\omega_1^2 + \omega_1\omega_2 + \omega_2^2).
\]
Moving $L_2( \eta , \omega)$ to the right-hand side, we obtain $\omega_1^2 + \omega_2^2 - 2\omega_1\omega_2 = 0$.  In other words, $\omega_1  = \omega_2$, which contradicts our setting in (\ref{stat: defnZ}). This example will serve as the base case for an induction argument in Theorem~\ref{T: dimk>2s-2}.
\end{example}

We mainly focus on $Z(\eta, \omega, s, k )$ whose associated matrix attains full rank. In fact, our proof of $\dim [ W(f) ] = \mu$ is essentially a verification of whether the associated matrix of $W(f) \isom Z( \delta, \alpha , n_1, r)$ is full rank. So we call $Z( \eta, \omega, s , k )$ \emph{\textbf{non-degenerate}} if its associated matrix (\ref{equ: assomatrix}) is of full rank. We also say $W(f)$ is \emph{\textbf{non-degenerate}} if its isomorphic image $Z( \delta, \alpha, n_1, r )$ is non-degenerate.

\begin{rem} \upshape \label{rem: equiv} $W(f) = Z( \delta, \alpha, n_1, r)$ is non-degenerate if and only if $\dim [ W(f) ] = \mu \iff \dim [Z(n_1, r) ] = r + 1 - n_1$, or the associated matrix $A$ of $Z(\delta, \alpha ; n_1, r)$ has full rank.
\end{rem}

\begin{example} \label{E: k=s=2} \upshape It is possible to construct degenerate $Z( \eta, \omega, s, k)$. Let $\eta = \omega = (1, -1) \in \CC^2$, we show $\dim [Z(\eta, \omega; 2, 2) ] = 2$. In this case, $k = s = 2$ and the associated matrix $A$ of $Z(\eta, \omega; 2, 2)$ has size $2 \times 3$. 
\[
	A =
	\begin{pmatrix}
		\eta_1 & \omega_1 \eta_1 - 1 & \omega_1( \omega_1 \eta_1 - 2) \\
		\eta_2 & \omega_2 \eta_2 - 1 & \omega_2( \omega_2 \eta_2 - 2)
	\end{pmatrix}
	\xRightarrow[\eta_2 = \omega_2 = -1]{\text{put } \eta_1 = \omega_1 = 1}
	A = 
	\begin{pmatrix}
	1 & 0 & -1\\
	-1 & 0 & 1
	\end{pmatrix}
	\sim
	\begin{pmatrix}
		1 & 0 & -1\\
		0 & 0 & 0
	\end{pmatrix}
	\Lrarr \rank A = 1 < 2. 	
\]
Since $A$ does not attain full rank, $Z(\eta, \omega, 2, 2)$ is degenerate by Remark~\ref{rem: equiv}.
\end{example}

\section{Zarhin's idea} \label{sec: ZarIdea}


Zarhin suggests that the Chinese reminder theorem (CRT) can be used to show nondegeneracy of $W(f)$.  Following his idea, we prove a corollary of Theorem~\ref{T: mainthm}.\ref{item: equdim} that 
\[
	\text{ If } r  \geq 2n_1 -1 , \text{ then } \dim [ W(f) ] = \mu.  \tag{\ref{sec: ZarIdea}.1} \label{stat: r>2n_1-1}
\]
In the second part, we generalize Zarhin's idea to claim 
\[
	\text{ If } k \geq 2s -1, \text{ then } \dim [ Z( \eta, \omega, s, k )] = k + 1 -s.  \tag{\ref{sec: ZarIdea}.2} \label{stat: k>2s-1}
\]
using Hermit interpolation. Recall Lemma~\ref{L: importL} says $W(f) \isom Z(\delta, \alpha ; n_1, r)$. So, (\ref{stat: r>2n_1-1}) is an instance of (\ref{stat: k>2s-1}) if we consider the substitution $( \delta, \alpha, n_1 , r ) \Rarr ( \eta, \omega, s, k)$. We assume $r \geq 2n_1 -1$ throughout this section. By (\ref{equ: r}),  we know that $n_2 + (k_1-2) + \dots + (k_{N_3} -2)  \geq n_1 + 1$. 


\subsection{Applications of CRT}  \label{subsec: CRT} To begin with, we denote $R = \CC[x]$, $I = \langle p(x) \rangle$ the ideal in $R$ generated by polynomial $p(x)$, and define an auxiliary polynomial $A_f(x) := f_{\alpha} (x) f_{\beta} (x) f_{\gamma}^2(x)$. Also we write $I_r = \langle x -  r \rangle$ for each $r \in R(f)$. As a result, we can define a quotient space corresponding to $A_f$, 
\[
	\textstyle V(f) := \oplus_{i = 1}^{n_1} (R / I_{\alpha_i} ) \oplus_{j = 1}^{n_2} (R / I_{\beta_j} )  \oplus_{l = 1}^{N_3}  (R / I_{\gamma_l}^2).
\]
Since ideals $I_{\alpha_i}, I_{\beta_j}, I_{\gamma_l}$ are coprime in $R$, we can apply CRT to say that 
\[
	V(f) \isom R / \langle f_{\alpha} \rangle \oplus R / \langle f_{\beta} \rangle \oplus R / \langle f_{\gamma} \rangle^2 \isom R/ \langle A_f\rangle \text{ as $\CC$-vector spaces. }
\]
 It follows that $\dim [V (f) ]= \deg [ A_f (x) ]= n_1 + n_2 + 2N_3$. Next, we consider a map $\wt{\pi}: P_{n-2} \rarr V(f)$ given by
\[
	\wt{ \pi }( p(x) ) :=
	\begin{cases}
		(d_ip(x) - p'(x) )(\bmod( x- \alpha_i) )		&\text{ if $1 \leq i \leq n_1$,}\\
		p(x) (\bmod(x - \beta_j) )				&\text{ if $1 \leq j \leq n_2$,}\\
		p(x) (\bmod(x - \gamma_l)^2)			&\text{ if $1 \leq l \leq N_3$,}
	\end{cases}	
	\text{ where }
	 d_i = \frac{ f''( \alpha_i ) }{ f'( \alpha_i) } \; \forall \; 1 \leq i \leq n_1.
\]
Note $d_i$ is well-defined since $\alpha_i$ are simple roots of $f(x)$. Applying Lemma~\ref{L: importL}, we deduce the kernel of $\wt{ \pi }$ is exactly $W(f)$. So if we can show $\wt{ \pi }$ is an epimorphism, the first isomorphism theorem \cite[Theorem 3.3.16]{DFBasicAlg} tells us that $V(f) \isom P_{n-2} / W(f)$. In other words,
\[
	\dim [ W(f) ] = \dim  [ P_{n-2} ] - \dim [ V(f) ] = n-1 - (n_1 + n_2 + 2N_3).
\]
So, proof of (\ref{stat: r>2n_1-1}) reduces to the following claim.

\begin{theorem} \label{T: CRsurj} The map $\wt{\pi}: P_{n-2} \rarr V(f)$ is an epimorphism.
\end{theorem}

\begin{proof} Given $a_i, b_j \in \CC$ and $c_l(x)  \in P_1$, we want to find a polynomial $p(x) \in P_{n-2}$ such that 
\[
	d_i p(x) - p'(x) \equiv a_i (\bmod( x- \alpha_i) ) , \;  p(x) \equiv b_j ( \bmod(x - \beta_j) ), \text{ and } p(x) \equiv c_l(x) (\bmod(x - \gamma_l)^2 )
	\tag{$\ast$}
\]
where $1 \leq i \leq n_1, 1 \leq j \leq n_2$, and $1 \leq l \leq N_3$. Let $Q(x) := f_{\alpha}^2(x) f_{\beta} (x ) f_{\gamma}^2(x)$, and for every $1 \leq i \leq n_1$ we construct
\[
	h_i(x) := 
	\begin{cases}
	a_i x + \wt{a_i}			&\text{if $d_i \ne 0$,} \\
	 -a_ix				&\text{if $d_i = 0$,}
	\end{cases}
	\text{ where } 
	\wt{a_i} :=  \frac{ 2a_i }{ d_i }- \alpha_i a_i.
	\tag{\ref{T: CRsurj}.1}
	\label{equ: CRTh}
\]
Since ideals $I_{\alpha_i}^2$, $I_{\beta_j}$, and $I_{ \gamma_l }^2$ are coprime in $\CC[x]$, the partial fraction of $1 / Q(x)$ can be written in the form
\[
	\frac{1 }{ Q(x) }=	\sum_{i=1}^{n_1} \frac{ A_i(x) }{ (x- \alpha_i)^2 } + \sum_{j=1}^{n_2} \frac{ b_j' }{ x - \beta_j } + \sum_{s =1}^{N_3} \frac{ C_s (x) }{ (x - \gamma_s)^2 } 
\]
where $b_j' \in \CC, A_i(x)$ and $C_s(x)$ are in $P_1$. Let $\wt{A}_i(x)$ be the reminder of $A_i(x) h_i(x)$ dividing $(x - \alpha_i)^2$ from the Euclidean algorithm, and similarly $\wt{ C}_s(x)$ are obtained from $c_s(x) C_s(x)$ dividing $(x - \gamma_s)^2$. We know the polynomial 
\[
	p(x) := \sum_{i=1}^{n_1} \frac{ \wt{ A}_i (x) Q(x) }{ (x- \alpha_i)^2 } + \sum_{j=1}^{n_2} \frac{ b_j' b_j Q(x) }{ x - \beta_j } + \sum_{s =1}^{N_3} \frac{ \wt{ C}_s(x) Q(x) }{ (x - \gamma_s)^2 }  \tag{\ref{T: CRsurj}.2} \label{equ: CRTp}
\]
is of degree at most $K := 2n_1 + n_2 + 2N_3 -1$. Furthermore, applying CRT, we can say $p(x)$ in (\ref{equ: CRTp}) satisfies the congruence relation:
\[
	p(x) \equiv h_i(x) ( \bmod (x- \alpha_i)^2 ), \; p(x) \equiv b_j ( \bmod (x - \beta_j) ), \text{ and } p(x) \equiv  c_l(x) ( \bmod ( x - \gamma_l)^2 ).   \tag{\ref{T: CRsurj}.3} \label{equ: congr}
\]
Since $r \geq 2n_1 -1$, $n-2 \geq K \Rarr p(x) \in P_K \subseteq P_{n-2}$. So to show $p(x)$ in (\ref{equ: CRTp}) is what we expect, it suffices to prove 
\[
	d_ip(x) - p'(x) \equiv a_i ( \bmod(x - \alpha_i) ) \text{ for each } 1 \leq i \leq n_1.
\]
If $d_i = 0$, then $d_i p(x) - p'(x) = -p'(x)$. Using $h_i(x)$ constrctued from (\ref{equ: CRTp}), we know $p(x) \equiv (-a_i x ) ( \bmod (x - \alpha_i)^2 )$. So we can write $p(x) = -a_i x + q_i(x) (x - \alpha_i)^2 \text{ for some } q_i(x) \in \CC[x]$. It follows that $p'(x) = -a_i + \big[ q_i' (x) (x - \alpha_i) + 2q_i(x) \big] (x - \alpha_i) \Rarr -p'(x) \equiv a_i (\bmod( x -\alpha_i ))$. Thus ($\ast$) holds for $1 \leq i \leq n_1$ when $d_i = 0$. 

If $d_i \ne 0$, we consider polynomials $g_i(x) := d_ix - (1 + d_i \alpha_i)$. By this construction: 
\[
	g_i'(x) = d_i \text{ and } g_i(\alpha_i) = -1. \tag{\ref{T: CRsurj}.4} \label{equ: evg_i}
\]
Since $p(x) \equiv (a_i x + \wt{a_i} )( \bmod (x - \alpha_i) )$, $g_i(x) p(x) \equiv g_i(x) (a_i x + \wt{a_i} )( \bmod( x- \alpha_i )^2 )$. Again by the construction of $p(x)$ in (\ref{equ: CRTp}), 
\[
	g_i(x) p(x) = g_i(x) (a_i x + \wt{a_i} ) + q_i(x) (x - \alpha_i)^2. \tag{\ref{T: CRsurj}.5} \label{equ: CRTg_i}
\]
for some $q_i(x) \in \CC[x]$. Let $\wt{ g }_i(x)$ be the first derivative of $g_i(x) (a_i x + \wt{ a_i } )$. 
\begin{align*}
	\wt{g}_i(x) &= g_i'(x) (a_i x + \wt{a_i} ) + g_i(x) a_i = 2d_i a_i x + [ d_i \wt{a_i} - a_i( 1 + d_i \alpha_i )  ].
\end{align*}
Evaluation of the polynomial $\wt{ g }_i$ at $\alpha_i$ yields
\begin{align*}
	\wt{g}_i(\alpha_i) &= 2d_i a_i \alpha_i + d_i \wt{a_i} - a_i - d_i a_i \alpha_i = 2d_i a_i \alpha_i + d_i \left( \frac{2a_i }{d_i} -  a_i \alpha_i \right) - a_i - d_i a_i \alpha_i = a_i.
\end{align*}
Differentiating on both sides of (\ref{equ: CRTg_i}), we obtain
\[
	g_i'(x) p(x) + g_i(x) p'(x) = \wt{g}_i(x) + \Big[ 2q_i(x) + q_i'(x) (x - \alpha_i ) \Big] (x - \alpha_i).
\]
It follows that $g_i'(\alpha_i) p(x) + g_i(\alpha_i) p'(x) \equiv \wt{g}_i(\alpha_i) (\bmod( x - \alpha_i ) )$. We know $\wt{g}_i(\alpha_i) = a_i$ and $g'_i(x) = d_i, g_i(\alpha_i) = -1$ by (\ref{equ: evg_i}). Therefore $d_i p(x) - p'(x) \equiv a_i (\bmod( x - \alpha_i ) )$. Finally, it's trivial to check that $\wt{\pi}$ is an $\CC$-linear map.
\end{proof}


\subsection{Applications of Hermite interpolation} \label{subsec: partialn_1>=2r-1} We begin with a statement of Hermite interpolation that fits into the context of $Z( \eta, \omega, s , k)$. 

\begin{theorem}[Hermite Interpolation] \label{T: Hermite} Let $y = (y_1, y_2, \dots, y_s) \in \CC^s$ be given. Then there exists a unique nonzero polynomial $h(x) \in Z( \eta, \omega ; s , 2s-1)$.
\end{theorem} 

The polynomial constructed in Theorem~\ref{T: Hermite} is an instance of Hermite interpolation, which involves the construction of polynomials and its derivatives with prescribed values at each point \cite[\S4.1.2]{VVPraso}. As a result, we prove the following corollary.

\begin{corollary} \label{C: HCn_1} The map $\ev_s: Z( \eta, \omega ; s , 2s-1) \rarr \CC^s$ given by $\ev_s (h) =  (h(\omega_1), h(\omega_2), \dots, h(\omega_{s}) )^T$ is a well-defined vector space isomorphism. 
\end{corollary}

\begin{proof} First note, $\ev_s$ is bijective from Theorem~\ref{T: Hermite}. To verify linearity, let $h, g \in Z( \eta, \omega ; s , 2s-1)$ and $c \in \CC$. Since both $\CC^s$ and $\CC[x]$ are vector spaces, 
\begin{align*}
	\ev_s (h ) + c \ev_s (g) &= ( (h + cg)(\omega_1), (h + cg)(\omega_2), \dots, (h + cg)(\omega_{s} ) )^T = \ev_s ( h+ cg).
\end{align*}
Since the map $\ev_s$ is well-defined, $Z( \eta, \omega ; s , 2s-1)$ is isomorphic to $\CC^s$ via $\ev_s$.
\end{proof}

\subsection*{Proof of statement (\ref{stat: k>2s-1})} Let $k \geq 2s - 1$ be given. As in Corollary~\ref{C: HCn_1}, the evaluation map $\ev_{s, k}: Z( \eta, \omega ; s , k) \rarr \CC^{s}$ given by $q(x) \mapsto (q(\alpha_1), q(\alpha_2), \dots, q(\alpha_{s}) )^T$ is again linear. In addition, $\ev_s$ is surjective because $\CC^s \isom Z(\eta, \omega; s, 2s-1)$ embeds into $Z( \eta, \omega ; s , k)$ as a subspace via Proposition~\ref{P: natEmbed}. By the first isomorphism theorem \cite[Theorem 3.3.16]{DFBasicAlg}, $Z( \eta, \omega ; s , k) / \ker \negthinspace (\ev_{s, k}) \cong \CC^{s}$. So $Z( \eta, \omega ; s , k) \cong \ker \negthinspace (\ev_{s, k}) \oplus \CC^{s}$, which implies $\dim  Z( \eta, \omega ; s , k) = \dim [\ker \negthinspace (\ev_{s, k}) ] + \dim \CC^{s} = \dim [\ker \negthinspace (\ev_{s, k}) ]+ s$. To compute the dimension of the kernel of $\ev_{s, k}$, recall that $\ker \negthinspace (\ev_{s, k}) : = \{ q(x) \in Z( \eta, \omega ; s , k) \mid q(\omega_i) = 0 \text{ for every } 1 \leq i \leq s, i \in \ZZ_+\}$. Observe from (\ref{stat: defnZ}), that $q(x) \in \ker ( \ev_{s, k} )$ if and only if $q(x), q'(x)$ vanish at all $\omega_i$s. Since $\omega_i \ne \omega_j$ for all $i \ne j$, $q(x) \in \ker ( \ev_{s, k} )$ if and only if it is divisible by $\prod_{i = 1}^{s} (x - \omega_i)^2 =: \Omega(x)$. It follows that 
\[
	\ker \negthinspace (\ev_{s, k}) = \{ g(x) \Omega(x) \mid g(x) \in \CC[x], \deg g \leq k - \deg \Omega \} \Lrarr  \dim [\ker \negthinspace (\ev_s) ] = (k - 2s) + 1.
\]
Therefore, $\dim  [Z( \eta, \omega ; s , k) ] = \dim [ \ker \negthinspace (\ev_{s, k}) ] + s = (k- 2s + 1) + s = k + 1 - s$.


\section{Reduction of associated matrix}  \label{sec: RAM} 

The main result we are going to prove in this section is that degenerate $Z( \eta, \omega, s, k)$s, where $k \geq 2s -2$, are restricted in the sense that $\eta_i = f_{\omega} ''( \omega_i ) / f_{\omega}' ( \omega_i )$ for all $1 \leq i \leq s$ where $f_{\omega}  := \prod_{i=1}^s (x - \omega_i)$. As an application, we prove $\dim [ W(f) ] = \mu$ whenever $r = 2n_1 - 2$. Combining Example~\ref{E: s=2}, Corollary~\ref{C: n_1=0} and (\ref{stat: r>2n_1-1}), we complete the proof of Theorem~\ref{T: mainthm}.\ref{item: equdim} except the last case $(n_1, r) = (3, 3)$, which is handled in \S\ref{sec: counterEX}. In order to prove our main results, we first present the following induction step.

\begin{lemma}\label{L: RAM} Assume $k \geq s + 1$ and let $Z(s +1, k) := Z(\eta, \omega ; s +1, k)$, and $Z(s, k-2) = Z(\wt{ \eta }, \wt{\omega}; s, k-2)$ where $\wt{ \omega }$ is obtained from $\omega$ by ignoring $\omega_i$ and $\wt{ \eta }_j = \eta_j - [ 2 /( \omega_j - \omega_i ) ]$ for all $j \ne i$. If $Z(s +1, k)$ degenerates, so does $Z(s, k-2)$. 
\end{lemma}

\begin{proof} Let $\ev_i: Z( s + 1, k ) \rarr \CC$ be the evaluation map $p(x) \mapsto p( \omega_i )$. We write $\wt{ Z }_i$ for the kernel of $\ev_i$. Then
\[
	\wt{ Z }_i = \{ p \in Z( s+1 , k) : p'( \omega_i ) = p( \omega_i ) = 0 \} =  \{ p \in Z(s +1, k) : (x - \omega_i)^2 \text{ divides } p \}.
\]
So we get an inclusion $\rho: \wt{ Z }_i \hrarr P_{k-2}$ defined via $p(x) \longmapsto p(x) / (x- \omega_i)^2$. We claim $\Ima \rho =Z(s, k-2)$. For every $p \in \wt{ Z }_i$, write $\wt{ p } (x) = ( \rho p ) ( x )$, we have $p( x) = (x- \omega_i)^2 \wt{ p } ( x)$. Differentiating on both sides, we obtain $p '( x ) =  2 ( x - \omega_i )  \wt{ p } ( x ) + (x -\omega_i)^2 \wt{  p } '( x )$. Substituting $x = \omega_j$ in this equation, we obtain
\[
	2( \omega_j - \omega_i ) \wt{ p }( \omega_j ) + ( \omega_j - \omega_i)^2 \wt{ p }' ( \omega_j )  = p'( \omega_j ) = \eta_j p( \omega_j ) = \eta_j ( \omega_j - \omega_i )^2 \wt{ p } ( \omega_j ).
\]
Dividing $(\omega_j - \omega_i)^2$ on both sides, we obtain $\wt{p } '( \omega_j )  =\wt{ \eta }_j \wt{p } ( \omega_j )$. So the condition $p'( \omega_j ) = \eta_j p( \omega_ j)$ is equivalent to $\wt{ p } ' ( \omega_j ) = \wt{ \eta }_j \wt{ p } ( \omega_j )$ for $p \in \wt{ Z }_i$ and $j \ne i$. Hence $\rho$ embeds $\wt{ Z }_i$ isomorphically onto $Z(s, k-2)$. Since the map $\ev_i$ is linear,
\[
	\dim [ Z(s +1, k ) ] - \dim [ Z(s, k-2) ] = \dim [ Z(s+1, k) / \wt{ Z } ] \leq \Ima \ev_i =1.
\]
If $Z(s +1, k)$ degenerates, then $\dim [ Z(s+1, k) ] > k +1 -(s+1)$. Combining this with the above inequality, $Z(s, k-2)$ is degenerate because $\dim [ Z(s, k-2) ] \geq \dim [ Z(s+1, k) ] -1 > (k-2) +1 -s$. 
\end{proof} 

\begin{rem} \label{rem: RAM} \upshape The map $\ev_i: Z(s +1, k) \rarr \CC$ is onto when $k \leq 2s +1$.  Suppose to the contrary, then we have $p( \omega_i ) = 0 = p' ( \omega_i )$ for all $p \in Z(s+1 , k)$ and $1 \leq i \leq s+1$. Hence $p(x)$ is divisible by the polynomial $\prod_{i=1}^{s+1} (x - \omega_i)^2$, which is of degree $2(s+1)$. This contradicts the fact that $\deg p \leq k \leq 2s +1$. As a result, $\dim[ Z(s +1, k) ] = 1 + \dim [ Z(s , k-2) ]$. Hence both $Z(s +1, k)$ and $Z(s, k-2)$ degenerate simultaneously.
\end{rem}

\begin{theorem} \label{T: dimk>2s-2} Given $\eta, \omega \in \CC^s$ with $s \geq 2$, $k \geq 2s -2$. If $Z(s, k) = Z(\eta, \omega; s, k)$ is degenerate then $\eta_i = f_{ \omega } '' ( \omega_i ) / f_{\omega}' ( \omega_i )$ for all $i = 1, \dots, s$ where $f_{\omega} := \prod_{i=1}^s ( x- \omega_i )$.
\end{theorem}

\begin{proof} We prove by induction on $s$. Recall (\ref{stat: s=2}) of Example~\ref{E: s=2} shows that degenerate $Z( 2, k)$ must satisfy $\eta_1 = 2 / ( \omega_1 - \omega_2) = - \eta_2$ when $k \geq 2$. This checks the base case for $s =2$. 

Now suppose $Z( s, k )$ is degenerate when $k \geq 2s-2$. By Lemma~\ref{L: RAM},  $\wt{ Z }( s-1, k-2 ) := Z(\wt{\eta}, \wt{\omega}; s-1, k-2)$ also degenerates with $\wt{ \eta }_i = \eta_i -[  2/ ( \omega_i - \omega_{ s } ) ]$ and $\wt{ \omega}_i = \omega_i$ for all $1 \leq i \leq s-1$. Note that the assumption $k \geq 2s -2$ implies $k -2 \geq 2(s-1) -2$. So we can apply induction hypothesis on $\wt{Z } (s-1, k-2)$ to conclude for each $i =1, \dots s-1$,
\[
	\wt{ \eta }_i = \sum_{j \ne i }^{s-1} \frac{2}{ \omega_i - \omega_j } \Lrarr \eta_i = \sum_{j \ne i}^{s-1} \frac{1}{ \omega_i - \omega_j } + \frac{2}{ \omega_i - \omega_{s} } = \sum_{j \ne i}^{s} \frac{2}{ \omega_i - \omega_j } = \frac{ f_{\omega}''( \omega_i ) }{ f_{\omega}'( \omega_i ) }.
\]
The last equality follows from (\ref{equ: polyalg}). Let $\wh{ \omega }, \wh{ \eta }$ be $(s-1)$-tuples where $\wh{ \omega }$ is defined from $\omega$ by ignoring $\omega_1$ and $\wh{ \eta }$ defined via $\wh{ \eta }_i := \eta_i  - [ 2 / ( \eta_i - \eta_1 ) ]$ for all $i =2, \dots, s$. Applying Lemma~\ref{L: RAM} to $Z( \wh{ \eta }, \wh{ \omega }, s-1 , k-2)$, we obtain $\eta_{s} = f_{\omega}'' ( \omega_{s} )  / f_{\omega}' ( \omega_{s} )$. 
\end{proof}


\subsection{Proof of Theorem~\ref{T: mainthm}.\ref{item: equdim} (when $r = 2n_1 -2$)}  \label{subsec: r>=2n_1-2} Recall in (\ref{equ: r}) that 
\[
	\textstyle r = n_1 + (n_2 -2) + \sum_{s=1}^{N_3} (k_s -2).
\]
If $r =2n_1 -2$, then 
\[
	 \textstyle n_1 = n_2 + \sum_{s=1}^{N_3} (k_s -2) \geq n_2 + \sum_{s=1}^{N_1} 1 = n_2 + N_3. \tag{\ref{subsec: r>=2n_1-2}.1}
\]
Recall the rational function $d(x)$ introduced in (\ref{equ: d_f}). We denote
\[
	\wt{d}(x)  = \wt{d }_f(x) := d(x) - \frac{ f''_{\alpha}(x) }{ f'_{\alpha}(x) } = \sum_{i = 1}^{n_2} \frac{3}{x - \beta_i} + \sum_{j = 1}^{N_3} \frac{2(k_j - 1)}{x - \gamma_j}. \tag{\ref{subsec: r>=2n_1-2}.2} \label{equ: dwt}
\]
Since $\wt{d}(x)$ is a rational function, let $h(x)$ be the numerator of $\wt{d}(x)$ (in lowest terms).

Since $W(f)$ does not vanish (i.e. \negmedspace $n_2 \geq 2$ or $N_3 \geq 1$ by \cite[Theorem~1.5.(ii)]{Zar12}), $\wt{d}(x)$ is not identically zero. So is the polynomial $h(x)$. By (\ref{subsec: r>=2n_1-2}.1), $\deg h \leq n_2 + N_3 -1 \leq n_1 -1$. So $h(x)$ cannot vanish at more than $n_1 - 1$ points by the fundamental theorem of algebra. Suppose to the contrary that $W(f) \isom Z(\delta, \alpha ; n_1, r)$ is degenerate. By Theorem~\ref{T: dimk>2s-2}, 
\[
	 \delta_i = d(\alpha_i) = \sum_{j \ne i}^{n_1} \frac{2 }{ \alpha_i - \alpha_j } \Lrarr \wt{ d } ( \alpha_i ) =  \delta_i -  \frac{ f_{\alpha}''( \alpha_i ) }{ f_{\alpha}'( \alpha_i ) }  = 0.
\]
So $h(x)$ also vanishes at all $\alpha_i$s. This contradicts the fact that  $\deg h \leq n_1 -1$. So all $W(f)$s that are isomorphic to $Z( \delta, \alpha, n_1, r)$ with $r = 2n_1 -2$ are non-degenerate. These spaces have dimension $\mu = r + 1 -n_1$ from Remark~\ref{rem: equiv}.


\section{Examples of degenerate $W(f)$} \label{sec: counterEX}

In this section we will prove the last case $(n_1, r) = (3, 3)$ of Theorem~\ref{T: mainthm}.\ref{item: equdim} and check that the boundedness condition on $n_1$ in Theorem~\ref{T: mainthm}.\ref{item: equdim} for $W(f)$ to be non-degenerate is necessary by constructing three types of explicit examples. (i.e. Theorem~\ref{T: mainthm}.\ref{item: noneg})

\begin{theorem} \label{T: counterEX} \upshape If $f (x) / f_{\alpha} (x ) \ne x^4$ then $W(f) \isom Z(4, 4)$ is degenerate (i.e., $\dim Z(4, 4) = 2$) with $\alpha_i$s appropriately chosen.
\end{theorem}

We outline the major steps in the construction of degenerate $W(f)$s:

\begin{enumerate}[label=(\roman*)]
	\item Modulo Proposition~\ref{P: ax+b},  there are only four types of polynomial $f$ such that $W(f) \isom Z(k, k), \forall  \; k \geq 3$.  \label{s: n_1=r4type}
	\item Recall in (\ref{equ: dwt}) $\wt{ d } (x) := d(x) - ( f_{\alpha} ''  / f_{\alpha} ' )( x ) \in \CC(x)$. As a rational function $\wt{ d }(x) = p(x) / q(x)$ where $p, q \in \CC[x]$ with $\deg p = \deg q - 1 = n_2 + N_3 -1$, and $q(x) = f_{\beta} ( x ) f_{\gamma } (x)$. Let $a$ be the leading coefficient of $p$. We claim $a \in \ZZ \cap [ 3n_2 + 4N_3, \infty)$.   \label{s: leadcoeff4dNum}
	\item For each pair $(n_1, r) \in \ZZ \times \ZZ$ satisfying $0 < r < 2n_1$, let $\ev_i: Z(n_1, r) \rarr \CC$ be the evaluation map $p(x) \mapsto p( \alpha_i )$. From Remark~\ref{rem: RAM} and Lemma~\ref{L: RAM}, $\ev_i$ is onto and has  kernel isomorphic to $Z(n_1, n_1 -2)$, which naturally sits in $Z(n_1, n_1-1) \isom W(f_i)$ where $f_i(x) = f(x) / (x- \alpha_i)$. \label{s: ev_ionto}
	\item If $n_1 = r = 3$, then $W(f)$ is non-degenerate. (i.e., $\dim_{\CC} Z(3,3) =1$) \label{s: Z(3,3)nondegen}
	\item If $f(x)$ is of one of the following form then the inclusion $Z(3, 2) \hrarr Z(3, 3) \isom W(f)$ is an isomorphism. \label{s: Z(3,2)isoZ(3,3)}
	\[
		(x^2-1)^2\cdot  \left(x^3 - \frac{1 }{ 3 } x \right),  (x^2-1)^3\cdot  \left(x^3 + \frac{ 3 }{ 11 } x \right), x^2(x-1)^3 \cdot \left( x^3 - \frac{ 15 }{ 11 } x^2 + \frac{ 6 }{ 11 } x - \frac{ 2 }{ 33 } \right).
	\]
\end{enumerate}

\subsection*{Proof of Theorem~\ref{T: counterEX} modulo \ref{s: n_1=r4type}$\sim$\ref{s: Z(3,2)isoZ(3,3)}.} Given $f / f_{\alpha} \ne x^4$ with $n_1= 4$, consider the evaluation homomorphism $\ev_4: Z(4, 4) \rarr \CC^1$. $\ev_4$ is onto by step \ref{s: ev_ionto}. Moreover, we have $\ker( \ev_4 ) \isom \wt{ Z } (3, 2 ) \hrarr \wt{ Z }( 3, 3 ) \isom W( f_4 )$ where $f_4(x) = f(x) / (x -\alpha_4)$. Hence $\dim Z(4, 4 ) = \dim \CC + \dim \ker(\ev_4) = 1 + \dim \wt{ Z } ( 3, 2 )$. On the other hand, $\dim W( f_4 ) = 1$ by step \ref{s: Z(3,3)nondegen}. So if $f_4$ is of the form in step \ref{s: Z(3,2)isoZ(3,3)}, we get $\dim \wt{ Z }( 3, 2 ) = \dim W(f_4) = 1 \Rarr \dim[ Z(4, 4) ] =2$, which shows $W(f) \isom Z(4, 4)$ is non-degenerate. 

\subsection{Proof of step \ref{s: n_1=r4type} and \ref{s: leadcoeff4dNum}}To begin with \ref{s: n_1=r4type}, let $k \geq 3$, we have for $n_1 = r = k$, 
\[
	r = n_1 \Rarr  n_1 + n_2 + N_3 -2 \leq (n-2) - (n_2 + 2N_3)   = r = n_1 \Rarr n_2 + N_3 \leq 2.
\]
If $W(f) \ne 0$, we must have $1 \leq n_2 + N_3$. Hence for $W(f) = Z(k , k)$, we have $n_2 + N_3 =1$ or 2. In the first case, $N_3 =1, n_2 = 0$ otherwise $f$ won't be divisible by a square of a quadratic polynomial. Hence $f(x) / f_{\alpha} ( x ) = x^m$ for some $m \geq 3$ integer. To compute $m$, notice that  
\[
	\deg f = m + n_1 \Rarr r = (m + n_1 -2) - (0 + 2 \cdot 1) = (m-4) +n_1 = n_1 \Rarr m = 4.
\]
Similarly, if $n_2 + N_3 =2$, then we deduce
\[
	f / f_{\alpha} =
	\begin{cases}
		(x^2- 1)^2  &\text{ when $n_2 =2$, $N_3 =0$}; \\
		x^2(x-1)^3	&\text{ when $n_2 = N_3 =1$}; \\
		(x^2-1)^3	&\text{ when $n_2 = 0$, $N_3 =2$}.
	\end{cases}
\]
In short, to study $Z(k , k ) \isom W(f)$, we only need to consider four special types of $f$. In other words, $\wt{ d } ( x ) = d(x) - (f_{\alpha} ''  / f_{\alpha}' )(x)$ is one of the following:
\[
	\wt{ d } (x) = \frac{ ax + b }{  x^2 + c x + d } \text{ with }
	\begin{pmatrix}
	a & b\\
	c & d
	\end{pmatrix}
	=
	\begin{pmatrix}
	6 & 0 \\
	0 & 0
	\end{pmatrix}
	,
	\begin{pmatrix}
	6 & 0 \\
	0 & -1
	\end{pmatrix}
	,
	\begin{pmatrix}
	7 & -3 \\
	-1 & 0
	\end{pmatrix}
	,
	\begin{pmatrix}
	8 & 0 \\
	0 & -1
	\end{pmatrix}.
	 \tag{i.1} \label{equ: dclassk=s}
\]
By Step \ref{s: n_1=r4type}, we can say the following.

\begin{proposition} Suppose that $n_1 = r$. Then $f$ has distinct multiple roots if and only if $f(x) \ne x^4  f_{\alpha} ( x)$ modulo certain affine change of coordinates $x \mapsto \lambda x + \mu$. In particular, Theorem~\ref{T: counterEX} is equivalent to Theorem~\ref{T: mainthm}.\ref{item: noneg}.
\end{proposition}

To show \ref{s: leadcoeff4dNum}, we know by common denominator 
\begin{align*}
	\wt{ d}_f (x) &= \sum_{i=1}^{n_2} \frac{ 3 }{ x - \beta_i } + \sum_{j =1}^{N_3} \frac{ 2(k_j-1) }{ x - \gamma_j } = \frac{ 1 }{ f_{\beta} ( x ) f_{\gamma} ( x ) } \left[ 3  \sum_{i=1}^{n_2} \prod_{l \ne i} (x - \beta_l) + 2 \sum_{j=1}^{N_3} (k_j -1) \prod_{l \ne j}  ( x - \gamma_l )  \right]
\end{align*}
Since each product $\prod_{l \ne i}^{n_2} (x - \beta_i) , \prod_{l \ne j}^{N_3} (x-  \gamma_l)$ is monic, we can calculate the leading coefficient of $p(x)$ as the following. 
\[
	\textstyle a = 3 \sum_{i=1}^{n_2} 1  + 2 \sum_{j=1}^{N_3} (k_j-1)  = 3n_2+ 2 \sum_{j=1}^{N_3} (k_j -1)  \in \ZZ
\] 
In particular, since each $k_j \geq 3$, we conclude $a \geq 3 n_2 + 2 \sum_{j=1}^{N_3} 2  = 3n_2 + 4N_3$. 

\subsection{Proof of Step \ref{s: Z(3,3)nondegen}} \label{subsec: proofZ(3,3)nondegen} Consider the following symmetric rational function in two variables:
\[
	D(T_1, T_2) := \frac{ \wt{ d } ( T_1 ) - \wt{ d } (T_2) }{ T_1 - T_2 } -  \wt{ d }( T_1 ) \wt{ d } (T_2).  \tag{\ref{subsec: proofZ(3,3)nondegen}.2}
\]

\begin{lemma} \label{L: nonvanD} \upshape Let $D(T_1, T_2)$ be the same as (\ref{subsec: proofZ(3,3)nondegen}.2). Then $D(\alpha_1 ,\alpha_2 ) \ne 0$ modulo certain permutations of $\alpha_i \in \alpha$.
\end{lemma}  

\begin{proof} Let $\wt{ D } (T_1, T_2):= D(T_1, T_2) \cdot (T_1^2 + cT_1 + d)(T_2^2 + cT_2 + d)$. Note $\wt{ D } \in \CC[T_1, T_2]$ is the numerator of the rational function $D$. So $D(\alpha_1, \alpha_2)  =0 \iff \wt{ D } ( \alpha_1, \alpha_2 ) =0$. In addition, $\wt{ D }$ is a symmetric function. So $\wt{D }$ can be expressed in the multinomial basis $\{ T_1 + T_2, T_1T_2, 1 \}$,
\[
	\wt{ D }(T_1, T_2) = x_{11} T_1T_2 + x_{10} (T_1 + T_2) + x_{00}  \tag{\ref{L: nonvanD}.1}
\]
where $x_{11} = - a(a+1), x_{10} = - b(a +1), x_{00} = (ad-bc)$. Suppose to the contrary that $\wt{ D } ( \alpha_i, \alpha_j ) = 0$ for any $1 \leq i \ne j\leq 3$. Let $\vec{ x } = (x_{00}, x_{10}, x_{11} )$ then $\wt{ D } (\alpha_i, \alpha_j) =0$ can be viewed as a linear equation $A \vec{ x } = \vec{ 0 }$ where 
\begin{align*}
	A &=
	\begin{pmatrix}
		1 & \alpha_1 + \alpha_2 & \alpha_1 \alpha_2 \\
		1 & \alpha_1 + \alpha_3 & \alpha_1 \alpha_3 \\
		1 & \alpha_2 + \alpha_3 & \alpha_2 \alpha_3 
	\end{pmatrix}
	\Lrarr
	\det A
	=
	\begin{vmatrix}
		1 & \alpha_1 + \alpha_2 & \alpha_1 \alpha_2 \\
		0 & \alpha_3 - \alpha_2 & \alpha_1( \alpha_3 - \alpha_2 ) \\
		0 & \alpha_3 - \alpha_1 & \alpha_2( \alpha_3 - \alpha_1) 
	\end{vmatrix}  
	= (\alpha_3 - \alpha_1)(\alpha_3 - \alpha_2) ( \alpha_2 - \alpha_1)  \ne 0.
\end{align*}
Now we can multiply $A^{-1}$ to obtain $\vec{x} = A^{-1} (A x )  = \vec{ 0 }$. In particular $x_{11} = 0 \Rarr a = 0$ or $1$, but $a \in \{ 6, 7 ,8 \}$ by (\ref{equ: dclassk=s}). This is a contradiction.
\end{proof} 

As a consequence of this lemma, we can finish the proof that $\dim Z(3, 3) = 1$. Let $\ev_3: W(f) \isom Z(3, 3) \rarr \CC$ be the evaluation map $p(x) \longmapsto p( \alpha_3 )$. It suffices to check $\ker \ev_3 = 0$. By \ref{s: ev_ionto}, we need to show $\wt{ Z } ( 2, 1 )  = Z( \wt{ \delta }, \wt{ \alpha }, 2, 1)= 0$ in $W(g) = \wt{ Z } ( 2, 2 )$ where $\wt{ \delta } = ( d_g( \alpha_1), d_g( \alpha_2 ) ) , \wt{ \alpha } = ( \alpha_1, \alpha_2), \text{ and } g(x) = f(x) / (x- \alpha_3)$.
Since $g_{\alpha} ( x ) = (x- \alpha_1) (x - \alpha_2)$, we can write 
\[
	d_g( \alpha_1) = \wt{ d_g } (\alpha_1 ) + \frac{ 2 }{ \alpha_1 - \alpha_2 } , d_g( \alpha_2 ) = \wt{d_g}( \alpha_2) + \frac{ 2 }{ \alpha_2 - \alpha_1} \tag{\ref{sec: counterEX}.2}
\]
Moreover $g, f$ only differ by a simple root factor $(x- \alpha_3)$, so $\wt{d_g} ( x)$ coincides with one of the $\wt{ d} (x)$ in (\ref{sec: counterEX}.1). Let $\wt{ A }$ be the associated matrix of $\wt{ Z }( 2, 1)$, we prove $\det \wt{ A }  \ne 0$. Recall 
\[
	\wt{ A } = 
	\begin{pmatrix}
		d_g(\alpha_1) & \alpha_1 d_g(\alpha_1) - 1 \\
		d_g( \alpha_2) & \alpha_2 d_g( \alpha_2) -1
	\end{pmatrix}
	\Lrarr \det \wt{A } =
	(\alpha_2 - \alpha_1) d_g(\alpha_1)d_g(\alpha_2) - [ d_g(\alpha_1) - d_g(\alpha_2) ].
\] 
By (\ref{sec: counterEX}.2), we have
\begin{align*}
	d_g(\alpha_1) - d_g(\alpha_2) &= \wt{ d_g } ( \alpha_1 ) - \wt{ d_g } ( \alpha_2) + \frac{ 4}{ \alpha_1 - \alpha_2}= \wt{ d } ( \alpha_1 ) - \wt{ d } ( \alpha_2) + \frac{ 4}{ \alpha_1 - \alpha_2}, \\
	(\alpha_2-\alpha_1)d_g(\alpha_1)d_g(\alpha_2) &= \wt{ d } ( \alpha_1 ) \wt{ d }( \alpha_2) ( \alpha_2 - \alpha_1) + 2 [  \wt{d } ( \alpha_1) - \wt{ d }( \alpha_2) ] +  \frac{ 4 }{ \alpha_1 - \alpha_2 }.
\end{align*}
It follows that $\det \wt{ A } =  \wt{ d } ( \alpha_1 ) \wt{ d }( \alpha_2) ( \alpha_2 - \alpha_1) +  [ \wt{d } ( \alpha_1) - \wt{ d }( \alpha_2) ] =(\alpha_1 - \alpha_2) D(\alpha_1, \alpha_2)$. We know $\det \wt{ A } \ne 0$ from Lemma~\ref{L: nonvanD}. So, $\wt{ Z } ( 2, 1) = 0 \Rarr \dim Z(3, 3 ) \leq  \dim [ \Ima \ev_3 ] = 1$.

\subsection{Proof of Step \ref{s: Z(3,2)isoZ(3,3)}}  \label{subsec: Z(3,2)isoZ(3,3)}We turn into the case where $Z(4, 4) \isom W(f)$ with $f / f_{\alpha}   = (x^2-1)^2$. By \ref{s: ev_ionto}, we can assume $\ev_4: Z(4, 4) \rarr \AA^1$ is onto. Using part \ref{s: ev_ionto} we have $\ker (\ev_4) \isom \wt{ Z } ( 3, 2 ) \hrarr \wt{ Z } ( 3 , 3 ) \isom W( g) \text{ where } g = f / (x- \alpha_4)$. By part \ref{s: Z(3,3)nondegen}, $\dim W(g) = 1$. Without loss of generality, we can choose a basis $\{ \wt{ p } \}$ for $W(g)$ such that $\wt{ p }$ is monic. Observe if $\wt{ p } \equiv 1$, then the interpolation condition on $W(g)$ becomes $d_g( \alpha_i ) = 0$ for $i = 1, 2, 3$. Since $d_g(x) = [ g_{\alpha} ''( x ) /  g_{\alpha}' ( x ) ]+ [  6 x / ( x^2 -1) ]$, the condition $d_g( \alpha_i ) = 0$ is equivalent to the existence of a nonzero constant $\lambda \in \CC$ such that 
\[
	g_{\alpha}''( x) (x^2 -1) + 6x g_{\alpha}' (x) = \lambda g_{\alpha} ( x ). \tag{\ref{sec: counterEX}.3}
\]
Let $g_{\alpha} ( x ) := x^3 - e_1 x^2 + e_2 x - e_3$, where $e_i$ are elementary symmetric functions in $\alpha_1, \alpha_2, \alpha_3$. We check (\ref{sec: counterEX}.3) has a solution in $\alpha_1, \alpha_2, \alpha_3$. To begin with, $g_{\alpha} ' ( x ) = 3x^2 -2e_1 x + e_2 \text{ and } g_{\alpha}''( x ) = 6x - 2e_1$. Let $L(x)$ be the left-hand side of equation (\ref{sec: counterEX}.3). By direct computation $L(x) =24 x^3 - 14 e_1x^2 + 6(e_2-1) x + 2e_1$.
Comparing coefficients of $x^i$ between $L(x)$ and $\lambda g_{\alpha}( x)$ in (\ref{sec: counterEX}.3), we obtain 
\[
	\lambda = 24, -24 e_1 = -14 e_1, 24 e_2 = 6(e_2 -1),  -24 e_3 = 2 e_1.
\]
So the existence of (\ref{sec: counterEX}.3) is same as the existence of three-tuple $(e_1, e_2, e_3) = (0, -1/3, 0)$. But $(\alpha_1, \alpha_2, \alpha_3) = (0, 1/ \sqrt{ 3 }, -1 / \sqrt{ 3 } )$ solves the system $(e_1, e_2, e_3) = (0, -1/3, 0)$. Hence by taking $\alpha_i$s appropriately $\wt{ Z }( 3, 2)$ contains the basis $\wt{ p }$ for $W(g)$.  In conclusion, $\wt{ Z } (3, 2)$ is isomorphic to $W(g)$ because both are one-dimensional.

Similarly, if $f / f_{\alpha} = (x^2 -1)^3$, we need to prove the existence of cubic polynomial $g_{\alpha}$ such that 
\[
	(x^2 -1)  g_{\alpha}''( x ) + 8x g_{\alpha}' (x ) = \lambda g_{\alpha} ( x ). \tag{\ref{sec: counterEX}.4}
\]
By setting $g_{\alpha} ( x ) = x^3 -e_1 x^2 + e_2 x - e_3$ and comparing coefficients, the same argument yields $(e_1, e_2, e_3) = (0,  3/ 11, 0)$. Hence we get a solution $(\alpha_1, \alpha_2, \alpha_3 ) = (0, 3i / \sqrt{ 33 }, -3i / \sqrt{ 33 } )$. In the last case $f / f_{\alpha} = x^2(x-1)^3$, we solve $g_{\alpha} ( x ) = (x- \alpha_1) ( x-  \alpha_2 )( x- \alpha_3) = x^3 - e_1 x^2 + e_2 x - e_3$ for 
\[
	(x^2-x) g_{\alpha}''( x )  + (7x -3) g_{\alpha}' (x ) = \lambda g_{\alpha} ( x ). \tag{\ref{sec: counterEX}.5}
\]
In this case where $(e_1, e_2, e_3) =( 15/ 11,  6/ 11,   -2 /33)$, such a polynomial $g_{\alpha}$ (distinct roots) exists because $\disc( g_{ \alpha } ) = -5736/14641 \ne 0$.

\begin{rem} \label{rem: Z(4,5)nondegn} \upshape The proof of $\dim Z(3, 3) = 1$ can be generalized to show that $W(f) \isom Z(4,5)$ is also non-degenerate. In this case $1 \leq n_2 + N_3 \leq 3$. So $d(x) = p(x) / q(x)$ for some $p, q \in \CC[x]$ with $\deg p =  \deg q -1 = 2$. In the same manner as (\ref{L: nonvanD}.1) in Lemma~\ref{L: nonvanD}, the symmetric polynomial $\wt{D }$ is of the form:
\[
	\wt{ D } (T_1, T_2 ) = x_{22} (T_1T_2)^2 + x_{21} (T_1^2T_2 + T_1 T_2^2) + x_{20} (T_1^2+ T_2^2) + x_{11}T_1T_2 + x_{10} (T_1 + T_2 ) + x_{00}.
\]
Again if $\wt{ D } ( \alpha_i, \alpha_j )= 0$ for all $1 \leq i \ne j \leq 4$, set $\vec{ x } = (x_{00}, x_{10}, x_{11}, x_{20}, x_{21}, x_{22} )$ we obtain a linear system $A \vec{ x } = \vec{ 0 }$ where
\[
	A =
	\begin{pmatrix}
		1 & \alpha_1 + \alpha_2 & \alpha_1 \alpha_2 & \alpha_1^2 + \alpha_2^2 & \alpha_1^2 \alpha_2 + \alpha_1 \alpha_2^2 & \alpha_1^2 \alpha_2^2 \\
		1 & \alpha_1 + \alpha_3 & \alpha_1\alpha_3 & \alpha_1^2 + \alpha_3^2 & \alpha_1^2 \alpha_3 + \alpha_1 \alpha_3^2 & \alpha_1^2 \alpha_3^2 \\
		1 & \alpha_1 + \alpha_4 & \alpha_1 \alpha_4 & \alpha_1^2 + \alpha_4^2 & \alpha_1^2 \alpha_4 + \alpha_1 \alpha_4^2 & \alpha_1^2\alpha_4^2  \\
		1 & \alpha_2 + \alpha_3 & \alpha_2 \alpha_3 & \alpha_2^2 + \alpha_3^2 & \alpha_2^2 \alpha_3 + \alpha_2 \alpha_3^2 & \alpha_2^2\alpha_3^2 \\
		1 & \alpha_2 + \alpha_4 & \alpha_2 \alpha_4 & \alpha_2^2 + \alpha_4^2 & \alpha_2^2 \alpha_4 + \alpha_2 \alpha_4^2 & \alpha_2^2\alpha_4^2 \\
		1 & \alpha_3 + \alpha_4 & \alpha_3 \alpha_4 & \alpha_3^2 + \alpha_4^2 & \alpha_3^2 \alpha_4 + \alpha_3 \alpha_4^2 & \alpha_3^2\alpha_4^2 
	\end{pmatrix}
\]
Let $\tau \in S_4$ be a transposition. Observe for each $1 \leq i \ne j \leq 4$ the map $(\alpha_i , \alpha_j) \mapsto (\alpha_{ \tau ( i) } , \alpha_{ \tau (j) })$ either fixes or interchanges two pairs of distinct rows in $A$. So $\det A$ is a symmetric polynomial in $\ZZ[ \alpha_1, \alpha_2, \alpha_3, \alpha_4]$. On the other hand, $\deg ( \det A ) = 12$ is same as the degree of the discriminant $\prod_{1 \leq i \ne j \leq 4} ( \alpha_i - \alpha_j )^2$. Hence $\det A = \lambda \prod_{1 \leq i < j \leq 4} ( \alpha_i - \alpha_j)^2$ for some $\lambda \in \ZZ$. By evaluation at the point $(\alpha_1, \alpha_2, \alpha_3, \alpha_4) = (0, 1, -1, 2)$, we get $\lambda = -1$. As before we deduce $\vec{ x } = A^{-1} \vec{ 0 } = \vec{ 0 }$. In particular $x_{11} = - a ( a + 1) = 0 \Rarr a = 0$ or $-1$. However by \ref{s: leadcoeff4dNum} $a \in \ZZ_+$, which give rise to a contradiction. So the existence of a pair $(\alpha_1, \alpha_2)$ such that $\wt{D }( \alpha_1, \alpha_2) \ne 0$ is again established.

By the same manner as step \ref{s: Z(3,3)nondegen}, we can check $\dim Z(4, 5) = 2$. Let $\ev_{3,4}: W(f) \rarr \CC^2$ be the evaluation map $p(x) \longmapsto (p(\alpha_3), p(\alpha_4) )$. As in step \ref{s: ev_ionto}, we get $\ker ( \ev_{3, 4})  \isom Z ( \wt{ \delta } , \wt{ \alpha }, 2, 1 )  \hrarr \wt{ Z }( 2, 3 ) \isom W(g)$ where $g(x) = f(x) / [(x -\alpha_3) (x- \alpha_4) ]$, $\wt{ \delta } = ( d_g( \alpha_1 ) , d_g( \alpha_2 ) )$, and $\wt{ \alpha } = ( \alpha_1, \alpha_2)$. Let $\wt{ A }$ be the associated matrix of $\wt{ Z } ( 2, 1 ) = Z( \wt{ \delta } ,\wt{ \alpha }  , 2 , 1 )$, the exact same argument as step \ref{s: Z(3,3)nondegen} shows $\det \wt{A } = (\alpha_1 - \alpha_2) D( \alpha_1, \alpha_2) \ne 0$. Hence we conclude $\ev_{3, 4}$ is injective which implies $\dim W(f) \leq \dim \CC^2 = 2$.
\end{rem}

\appendix


\section{Supplement of Theorem~\ref{T: mainthm}.\ref{item: noneg}}  \label{sec: n_1=r=4}

Recall by Theorem~\ref{T: counterEX} (Theorem~\ref{T: mainthm}.\ref{item: noneg}) we only need to examine the case $f(x) =x^4 f_{\alpha} (x)$ in order to obtain a complete classification of $W(f)$ for $n_1 =4$. We use the notation from \S\ref{sec: counterEX} throughout this section. 

Again \S\ref{sec: counterEX}.\ref{s: ev_ionto} implies the map $\ev_4: Z(4, 4) \rarr \CC$ is onto. As in \S\ref{sec: counterEX}.\ref{s: ev_ionto}, $\ker( \ev_4 ) \isom Z( \wt{ \delta }, \wt{ \alpha }, 3, 2 ) \hrarr \wt{ Z } ( 3, 3 ) \isom W( g )$ where $\wt{ \delta } = ( d_g( \alpha_1), d_g( \alpha_2), d_g( \alpha_3) ), \wt{ \alpha } = (\alpha_1, \alpha_2, \alpha_3)$ and $g(x) =f(x) / (x- \alpha_4)$. As before, it suffices to show $\ker \ev_4 =0$. By step \S\ref{sec: counterEX}.\ref{s: Z(3,3)nondegen}, $\dim W(g) = 1$. Let $\{ \wt{ p } \} \subseteq W(g)$ be a basis, it is enough to show $\deg \wt{ p } = 3$. We may assume that $\wt{ p }$ is monic. We claim it is impossible for $\deg \wt{ p } < 3$. 

\underline{If $\deg \wt{ p } = 0$}, then $p \equiv 1 , p ' \equiv 0$. So the system $\wt{ p } ' ( \alpha_i ) = d_g( \alpha_i ) \wt{ p } ( \alpha_i ), 1 \leq i \leq 3$ is equivalent to $d( \alpha_i ) = 0$ for all $1 \leq i \leq 3$. This means $x g_{\alpha} '' ( x ) + 6 g_{\alpha}' ( x ) \text{ vanishes at } \alpha_1, \alpha_2, \alpha_3  \iff g_{\alpha} (x) \text{ divides }x g_{\alpha} '' ( x ) + 6 g_{\alpha}' ( x )$. Since $\deg g = 3$, $x g_{\alpha}'' ( x ) + 6 g_{\alpha}'( x)$ has degree at most 2, which cannot be divisible by $g_{\alpha}$.  

\underline{If $\deg \wt{ p } = 1$}, let $\wt{ p }(x) = x-r$. In this case, the interpolation condition $\wt { p } '( \alpha_i ) = d( \alpha_i ) \wt{ p } ( \alpha_i )$ becomes $1 = (\alpha_i -r) \wt{ d }( \alpha_i )$. It follows that  $r = \alpha_i - [1 / d( \alpha_i ) ]$ for $i =1 ,2 ,3$. In other words, for all $1 \leq i  \ne  j \leq 3$
\[
	\alpha_i - \frac{ 1 }{ d( \alpha_i ) } = \alpha_j - \frac{ 1 }{ d( \alpha_j ) } \iff (\alpha_i - \alpha_j) d( \alpha_i )d ( \alpha_j ) + [d ( \alpha_i ) - d( \alpha_j ) ] = 0.
\]
This implies $\wt{ D }( \alpha_i, \alpha_j ) = 0$ for all $i \ne j$ in $W(g)$, which is impossible by Lemma~\ref{L: nonvanD} in step \ref{s: Z(3,3)nondegen} of \S\ref{sec: counterEX}.

\underline{If $\deg \wt{ p } =2$}, then $\wt{ p } ( x ) = x^2 + a_1 x + a_0$ for nonzero constants $(a_1, a_0) \in \CC^2$. Our strategy is to rewrite the system $\wt{ p } '( \alpha_i ) = d( \alpha_i ) \wt{ p } ( \alpha_i )$ into a polynomial equation and comparing coefficients. The vanishing condition of $\wt{ p } '( \alpha_i ) - d_g( \alpha_i ) \wt{ p } ( \alpha_i ) = 0$ is equivalent to say the polynomial $R( g, \wt{ p } ) ( x ) :=    [ x g_{\alpha}''(x) + 6 g_{\alpha}' (x ) ] \wt{ p } ( x )  - x g_{\alpha}'' ( x ) \wt{p } '( x )$ is divisible by $g_{\alpha} ( x ) = (x- \alpha_1)(x- \alpha_2) (x - \alpha_3)$. As in \S\ref{subsec: Z(3,2)isoZ(3,3)}, we write 
\[
	g_{\alpha} (x) = x^3 -e_1 x^2 + e_2 x -e_3 ,  g_{\alpha}'(x) = 3x^2 - 2e_1 x + e_2, \text{ and } g_{\alpha} '' ( x ) = 6x - 2e_1.
\]
Combining this with $\wt{ p } (  x ) = x^2 + a_1 x + a_0, \wt{ p } '( x ) = 2x + a_1$, we expand $R( g, \wt{ p } ) ( x )$ in the monomial basis $\{ 1, x, \dots, x^4\}$:
\[
	R(g, \wt{ p} ) ( x ) = 18 x^4 + (21a_1 - 10e_1) x^3 + (24 a_0 + 4e_2 - 12 a_1 e_1) x^2 +  (5 e_2 a_1 - 14 e_1 a_0) x^1 + 6 e_2a_0.
\]
On the other hand, since $R( g , \wt{ p } )$ is a quartic polynomial divisible by $g_{\alpha} ( x )$, we can write
\[
	R( g , \wt{ p } ) ( x ) = 18 (x^3 - e_1 x^2 + e_2 x -e_3) ( x- r ) = 18 [ x^4  - (r +e_1) x^3 + (re_1 + e_2) x^2 - (r e_2 + e_3) x + re_3 ],
\]
where $r \in \CC$ is the remaining root except $\alpha_i$s. Comparing coefficients of $x^i$ in above expressions of $R(g, \wt{ p} )$, the existence of $\deg \wt{ p } = 2$ reduces to the existence of pairs $(a_0, a_1, r)$ with $a_0 a_1 \ne 0$ such that the following equation holds:
\begin{align*}
	18 r e_3 &= 6 e_2 a_0,   -18( re_2 + e_3 ) = 5 e_2 a_1 - 14 e_1 a_0,   \\
	18 (re_1 + e_2 ) &=  24 a_0 + 4 e_2 - 12 a_1e_1 ,  -18(r + e_1 ) = 21 a_1 - 10 e_1.
\end{align*}
Solving $r$ for each equation, we obtain
\[
	r = \frac{ a_0 e_2 }{ 3 e_3 } =  - \frac{ 7 }{ 6 } a_1 - \frac{ 4 }{ 9 } e_1 = - \frac{ 2 }{ 3 } a_1 + \frac{ 4 a_0 }{ 3 e_1 } - \frac{ 7 e_2 }{ 9 e_1 } = - \frac{ 5 }{ 18 }{ a_1 } + \frac{ 7 e_1 }{ 9 e_2 } a_0 - \frac{ e_3 }{ e_2 }.
\]
We can see from these equations that coefficients of $a_1$ are in $\QQ$. We would take $r  =a_0 e_2 / (3 e_3)$ and solve other three equations in terms of $a_1$ to simplify this system.
\[
	a_1 = - \frac{ 2 e_2 }{ 7 e_3 } a_0 - \frac{ 8 e_1 }{ 21 }  = - \frac{ 8  }{ 3 e_1 } a_0 + \frac{ 14 e_2 }{ 9 e_1 } - \frac{ 8 e_1 }{ 9 }  = - \frac{ 7 e_1 }{ 8 e_2 } a_0 + \frac{ 9 e_3 }{ 8 e_2 }  - \frac{ e_1}{ 2 }.
\]
Now we can set $a_1$ to one of three quantities on the right and solve the other 3 for $a_0$:
\[
	\frac{ 56 e_3 - 6 e_1 e_2 }{ 21 e_1 e_3 } a _0 = \frac{ 14 e_2 }{ 9 e_1 } - \frac{ 32 e_1 }{ 63 } \text{ and } \frac{ 49 e_1 e_3 - 16 e_2^2 }{ 56 e_2 e_3 } a_0 = \frac{ 9 e_3 }{ 8 e_2 } - \frac{ 5 e_1 }{ 42 }.  \tag{$\ast$}
\]
Hence the existence of $\deg \wt{ p } = 2$ can be guaranteed by a solution of $(a_0, \alpha_1, \alpha_2, \alpha_3)$ with $\alpha_i \ne \alpha_j$. Suppose such a 4-tuple $(a_0, \alpha_1, \alpha_2, \alpha_3)$ exists. If $a_0 = 0$ then the above system is equivalent to 
\[
	 e_2 = \frac{ 16 }{ 49 } e_1^2 \text{ and } e_3 = \frac{ 20 }{ 189 } e_1e_2 = \frac{ 320 }{ 9261 } e_1^3.
\]
Under this relation the polynomial $g_{\alpha}$ becomes:
\[
	g_{\alpha} ( x )  = x^3 - e_1 x^2 + \frac{ 16 }{ 49 } e_1^2 x  -  \frac{ 320 }{ 9261 } e_1^3.
\]
Since $\disc( g_{\alpha} )= 0$, $g_{\alpha} = (x- \alpha_1)(x- \alpha_2) (x- \alpha_3)$ has multiple roots, a contradiction. Because $a_0 \ne 0$, we can divide the two equations in $(\ast)$ to cancel $a_0$: 
\[
	 \frac{ 56 e_3 - 6 e_1 e_2 }{ 21 e_1 e_3 }  \cdot \left(  \frac{ 9 e_3 }{ 8 e_2 } - \frac{ 5 e_1 }{ 42 }  \right)  =   \frac{ 49 e_1 e_3 - 16 e_2^2 }{ 56 e_2 e_3 }  \cdot  \left(  \frac{ 14 e_2 }{ 9 e_1 } - \frac{ 32 e_1 }{ 63 }   \right)
\]
Under the condition that $(e_1, e_2, e_3) \ne (0, 0, 0)$ the equation above is equivalent to the vanishing of the polynomial: $h  := 27 e_3^2-  18 e_1e_2 e_3+ 4( e_2^3 + e_1^3 e_3) - e_1^2 e_2^2$.
But we can also view $h$ as a polynomial in $\alpha_1, \alpha_2, \alpha_3$. In fact explicit computation shows $h = h( \alpha_1, \alpha_2 , \alpha_3 )  =196 (\alpha_1- \alpha_2)^2(\alpha_1 - \alpha_3)^2 (\alpha_2 - \alpha_3)^2 = [14  \disc( g_{\alpha } ) ]^2  \ne 0$. From both cases, $(\ast)$ has no solution for $(a_0, \alpha_1, \alpha_2, \alpha_3)$ with $\alpha_i \ne \alpha_j$. Therefore we conclude when $f(x) = x^4 f_{\alpha} (x  )$,  $\dim \wt{ Z} ( 3, 2 ) = 0 \Rarr \dim Z( 4, 4 ) = 1$.

\begin{rem} \label{rem: allW4n_1=4} \upshape By Theorem~\ref{T: mainthm}, Remark~\ref{rem: Z(4,5)nondegn} and Appendix~\ref{sec: n_1=r=4}, we completely classify $W(f)$ when $n_1 = 4$. Indeed, let $n_1 = 4$. Theorem~\ref{T: mainthm}.\ref{item: equdim} says $\dim [ W(f) ] = r -3$ for all $r \geq 6$. Remark~\ref{rem: Z(4,5)nondegn} says $\dim [ W(f) ] = 2$ when $r = 5$. Appendix~\ref{sec: n_1=r=4} and Theorem~\ref{T: counterEX} together show that $r = 4$ implies $\dim [W(f) ] = 2$ or 1 depending on whether $f$ has distinct multiple roots. 
\end{rem}

\section*{Acknowledgement} This note has arisen from an attempt to answer questions suggested by Yuri Zarhin in connection with \cite{Zar12}. I would like to thank him for his questions, stimulating discussions, and interest in this paper. I am also grateful to his patience on reading several preliminary versions of this note and making extremely useful remarks. In addition, I would like to thank Xiyuan Wang, whose comments helped to improve the exposition.

\bibliography{FinalDraft}

\begin{thebibliography}{1}

\bibitem{BHI05}
G.~T. Buzzard, S.~L. Hruska, and Y.~Ilyashenko.
\newblock Kupka-{S}male theorem for polynomial automorphisms of {$\Bbb C^2$}
  and persistence of heteroclinic intersections.
\newblock {\em Invent. Math.}, 161(1):45--89, 2005.

\bibitem{DFBasicAlg}
D.~S. Dummit and R.~M. Foote.
\newblock {\em Abstract algebra}.
\newblock John Wiley \& Sons, Inc., Hoboken, NJ, third edition, 2004.

\bibitem{Gor16}
I.~Gorbovickis.
\newblock Algebraic independence of multipliers of periodic orbits in the space
  of polynomial maps of one variable.
\newblock {\em Ergodic Theory Dynam. Systems}, 36(4):1156--1166, 2016.

\bibitem{VVPraso}
V.~V. Prasolov.
\newblock {\em Polynomials}, volume~11 of {\em Algorithms and Computation in
  Mathematics}.
\newblock Springer-Verlag, Berlin, 2010.
\newblock Translated from the 2001 Russian second edition by Dimitry Leites,
  Paperback edition [of MR2082772].

\bibitem{Rees15}
E.~Rees.
\newblock On a paper by {Y}uri {G}. {Z}arhin.
\newblock {\em Eur. J. Math.}, 1(4):717--720, 2015.

\bibitem{Zar12}
Y.~G. Zarhin.
\newblock Polynomials in one variable and ranks of certain tangent maps.
\newblock {\em Math. Notes}, 91(3-4):508--516, 2012.
\newblock Translation of Mat. Zametki {{\bf{9}}1} (2012), no. 4, 539--550.

\bibitem{Zar13}
Y.~G. Zarkhin.
\newblock One-dimensional polynomial mappings, periodic points, and
  multipliers.
\newblock {\em Izv. Ross. Akad. Nauk Ser. Mat.}, 77(4):59--72, 2013.

\end{thebibliography}
\bibliographystyle{abbrv}

\end{document}